\documentclass[a4paper]{article}

\usepackage[margin=3cm]{geometry}  
\usepackage{amssymb, amsfonts, amsmath, amsthm, amsxtra}  
\usepackage{thmtools, thm-restate}  
\usepackage{graphics, graphicx}  
\usepackage{old-arrows}  
\usepackage{enumerate}  
\usepackage{authblk}  
\usepackage{hyperref}  
\usepackage{xcolor}  
\usepackage{xspace}

\definecolor{exodus}{RGB}{104, 109, 224}  
\definecolor{asbestos}{RGB}{127, 140, 141}  
\definecolor{crimson}{RGB}{162, 0, 37}
\definecolor{lightcyan}{RGB}{0, 172, 193}
\definecolor{lightpurple}{RGB}{126, 87, 194}

\hypersetup{
	colorlinks=true,
	citecolor=lightcyan,
	linkcolor=lightpurple,
	urlcolor=lightpurple
}

\def\FIGcarpet{1}  
\def\FIGleaf{1}  
\def\FIGgluing{0.8}  
\def\FIGrelaxed{1}  
\def\FIGsdm{1}  
\def\FIGjdis{1}  
\def\FIGmodproto{1}  
\def\FIGmod{1}  
\def\FIGgeomfull{1}  
\def\FIGgeom{1}  
\def\FIGliftscheme{1}  
\def\FIGleaflift{1}  
\def\FIGusm{1}  
\def\FIGminlift{1}  
\def\FIGgridlift{1}  

\def\arxiv{1}

\newcommand{\cc}[1]{\mathcal{#1}}  

\newcommand{\mtt}[2]{\texorpdfstring{#1}{#2}}  

\newcommand{\st}{:}  

\newcommand{\agset}[1]{\mathit{#1}}  

\newcommand{\U}{S}  
\newcommand{\pow}[1]{\cc{P}(#1)}  
\newcommand{\card}[1]{\vert #1 \vert}  

\DeclareMathOperator{\ftr}{\uparrow\!}  
\DeclareMathOperator{\idl}{\downarrow\!}  
\DeclareMathOperator{\upp}{\uparrow}  
\DeclareMathOperator{\downp}{\downarrow}  
\DeclareMathOperator{\dpp}{\updownarrow}  
\DeclareMathOperator{\mt}{\land}  
\DeclareMathOperator{\jn}{\lor}  

\newcommand{\cl}{\phi}  
\newcommand{\cs}{\cc{C}} 
\newcommand{\gen}{\mathrm{gen}}  

\newcommand{\imp}{\rightarrow}  
\newcommand{\is}{\Sigma}  

\declaretheorem[name=Theorem, style=plain]{theorem}
\declaretheorem[name=Lemma, style=plain]{lemma}
\declaretheorem[name=Corollary, style=plain]{corollary}
\declaretheorem[name=Proposition, style=plain]{proposition}
\declaretheorem[name=Question, style=plain]{question}
\declaretheorem[name=Definition, style=definition]{definition}
\declaretheorem[name=Example, style=definition]{example}

\declaretheorem[name=Remark, style=remark]{remark}
\declaretheorem[name=Claim, style=remark]{claim}

\usepackage{marginnote}
\usepackage[export]{adjustbox} 

\definecolor{update}{RGB}{58, 134, 255}
\definecolor{kira}{RGB}{255, 0, 110}
\definecolor{simon}{RGB}{251, 86, 7}
\definecolor{add}{RGB}{131, 56, 236}
\definecolor{todo}{RGB}{255, 190, 11}

\title{On the $E$-base of Finite Lattices: Semidistributive, Modular, and Geometric Lattices}%

\author[1]{Kira Adaricheva}
\author[2]{Simon Vilmin}

\affil[1]{Department of Mathematics, Hofstra University, \mbox{Hempstead, NY 11549, USA.}}
\affil[2]{Aix Marseille Univ, CNRS, LIS, Marseille, France.}

\begin{document}

\maketitle

\begin{abstract}
Implicational bases are a well-known representation of closure spaces and their closure lattices.
This representation is not unique, though, and a closure space usually admits multiple bases.
Among these, the canonical base, the canonical direct base as well as the $D$-base aroused significant attention due to their structural and algorithmic properties.
Recently, a new base has emerged from the study of free lattices: the $E$-base.
It is a refinement of the $D$-base that, unlike the aforementioned implicational bases, does not always accurately represent its associated closure space.
This leads to an intriguing question: for which classes of (closure) lattices do closure spaces have valid $E$-base?
Lower-bounded lattices are known to form such a class.
In this paper, we prove that for semidistributive lattices, the $E$-base is both valid and minimum.
We also characterize those modular and geometric lattices that have valid $E$-base.
Finally, we prove that any lattice is a sublattice of a lattice with valid $E$-base.

\ifx\arxiv\undefined
\else
\vspace{.5em}
\fi
\end{abstract}

\section{Introduction}

Implicational bases (IBs) are a well-known representation of finite closure spaces and their associated closure lattices. 
Implications go by many names in a broad range of fields, e.g., attribute implications in Formal Concept Analysis \cite{ganter2012formal}, functional dependencies in database theory \cite{maier1980minimum,mannila1992design}, Horn clauses in propositional logic \cite{crama2011boolean}, or join-covers in lattice theory \cite{freese1995free,gratzer2011lattice}.
Since a closure space admits several IBs in general, a rich development of implications, summarized in two surveys \cite{bertet2018lattices,wild2017joy}, emerged from this variety of perspectives.
Moreover, recent contributions \cite{adaricheva2025computing,berczi2024hypergraph,bichoupan2023independence} show that IBs are still an active topic of research.
Among possible IBs, two have attracted interest due to their structural and algorithmical properties.
First, the canonical (or Duquenne-Guigues) base \cite{day1992lattice,guigues1986familles,wild1994theory}, has a minimum number of implications and is reflected in any other IB by means of pseudo-closed and essential sets.
Second, the canonical direct base, surveyed in \cite{bertet2010multiple}, is based on minimal generators and is the smallest IB enjoying directness, being the property of computing closures with only one pass over the implications.

Translating the language of join-covers in lattices \cite{freese1995free} into implications, Adaricheva et al.~\cite{adaricheva2013ordered} recently introduced two new specimens in the zoo of IBs: the $D$-base and the $E$-base.
These bases find their origin in lattice theory where special relations between join-irreducible elements were developed to tackle notoriously hard problems.
The history begins at the cross-section of Whitman's conjecture\footnote{He asked his question in the same paper where he solved the fundamental problem of Birkhoff~\cite{birkhoff1935structure} about representing lattices by equivalence relations.} related to lattices of partitions \cite{whitman1946lattices} and J\'onsson's conjecture about the structure of finite sublattices of free lattices \cite{jonsson1962finite}.
On their way to solve the first problem, Pudl\`ak and T\r{u}ma used a special binary relation on join-irreducible elements, called the $C$-relation \cite{pudlak1974yeast}.
Only a few years later, they solved Whitman's problem in a celebrated paper \cite{pudlak1980every}.
While working on J\'onsson's conjecture, Day~\cite{day1979characterizations} used the $C$-relation, which became the $D$-relation when the solution to the problem---brought by Nation in his famous paper \cite{nation1982finite}---was incorporated to the monograph on free lattices of Freese, Je\v{z}ek~ and Nation \cite{freese1995free}.
There, the authors highlighted the connection between the $D$-relation and the so-called minimal join-covers that will later become the $D$-base.
In investigating lattices without $D$-cycles (cycles in the $D$-relation), they also unveiled an important subset of minimal join-covers along with a subset of the $D$-relation, which they called the $E$-relation. 
This subset of join-covers will become what is now known as the $E$-base.

The $D$-base is a subset of the canonical direct base that still enjoys directness as long as implications are suitably ordered.
It already proved its efficiency in a handful of data analysis applications \cite{adaricheva2023importance,adaricheva2017discovery,adaricheva2015measuring,nation2021combining} and its algorithmic aspects have been addressed in a recent paper~\cite{adaricheva2025computing}.
The $E$-base on the other hand is a refinement of the $D$-base obtained by retaining implications satisfying a minimality criterion within the closure lattice.
Borrowing tools from free lattices \cite{freese1995free}, especially the $D$-relation, Adaricheva et al.~\cite{adaricheva2013ordered} show that in closure spaces without $D$-cycles---closure spaces with lower-bounded lattices---the $E$-base is a valid IB. 
This encompasses in particular Boolean and distributive lattices where the $E$-base is trivially valid.
Furthermore, the main property of the $D$-base, its feature of directness, is still retained in the $E$-base of such spaces. 
Even more importantly, the number of implications in the (aggregated) $E$-base is the same as in the canonical base, making it minimum.
Yet, unlike the $D$-base and other mentioned IBs, the $E$-base does not always encode faithfully its closure space.
This prompts the next stimulating question, which constitutes the main motivation of our paper:

\begin{question} \label{ques:E-valid}
For what classes of (closure) lattices do closure spaces have valid $E$-base?
\end{question}

For such classes, the $E$-base could be a base of choice given that it is shorter than the $D$-base and may still possess the advantages the $D$-base has over other existing bases. 

In order to study this question, we first relate the $E$-base to (join-)prime elements in lattices on the one hand, and to the canonical base on the other hand, and more precisely to the related quasi-closed, pseudo-closed, and essential sets.
Leveraging from these relationships, we study three classical generalizations of Boolean and/or distributive lattices, namely semidistributive, modular, and geometric lattices.
Since all three classes do present $D$-cycles in general, the approach used in \cite{adaricheva2013ordered,freese1995free} for lower-bounded lattices cannot apply.
Regarding semidistributive lattices, our main result is that they have valid $E$-base.
\ifx\arxiv\undefined
\begin{theorem} \label{thm:E-valid-SD}
The (aggregated) $E$-base of a standard closure space with semidistributive lattice is valid and minimum.
\end{theorem}
\else
\begin{theorem}[restate=THMEvalidSD, label=thm:E-valid-SD]
The (aggregated) $E$-base of a standard closure space with semidistributive lattice is valid and minimum.
\end{theorem}
\fi

For modular and geometric lattices, the situation turns out to be different as the $E$-base may not be valid in general (see Examples~\ref{ex:fail-mod} and \ref{ex:fail-geom}).
However, we characterize lattices with valid $E$-base for both classes.

\ifx\arxiv\undefined
\begin{theorem} \label{thm:modular-Ebase}
The (aggregated) $E$-base of a standard closure space with modular lattice is valid if and only if for every essential set $C$ and any predecessor $C'$ of $C$, $\card{C' \setminus C_*} = 1$, where $C_*$ is the intersection of the predecessors of $C$.
\end{theorem}
\else
\begin{theorem}[restate=THMmodular, label=thm:modular-Ebase]
The (aggregated) $E$-base of a standard closure space with modular lattice is valid if and only if for every essential set $C$ and any predecessor $C'$ of $C$, $\card{C' \setminus C_*} = 1$, where $C_*$ is the intersection of the predecessors of $C$.
\end{theorem}
\fi

\ifx\arxiv\undefined
\begin{theorem} \label{thm:matroid-valid}
The (aggregated) $E$-base of a standard closure space with geometric lattice is valid if and only if all essential sets are incomparable.
\end{theorem}
\else
\begin{theorem}[restate=THMgeometric, label=thm:matroid-valid]
The (aggregated) $E$-base of a standard closure space with geometric lattice is valid if and only if all essential sets are incomparable.
\end{theorem}
\fi

As a byproduct, we obtain that geometric modular lattices do have valid $E$-base.
Moreover, we can use the correspondence between geometric lattices and matroids~\cite{oxley2006matroid} to initiate the study of matroids whose associated closure space have valid $E$-base.
In particular, Theorem~\ref{thm:matroid-valid} yields as a corollary that closure spaces of binary matroids do have valid $E$-base. 

Then, looking at Question~\ref{ques:E-valid} from another perspective we build an embedding proving that any lattice is a sublattice of a lattice with valid $E$-base.

\ifx\arxiv\undefined
\begin{theorem} \label{thm:sublattice}
For any standard closure space $(\U, \cl)$ with lattice $(\cs, \subseteq)$, there exists a standard closure space $(T, \psi)$ with lattice $(\cc{F}, \subseteq)$ such that $(\cs, \subseteq)$ is a sublattice of $(\cc{F}, \subseteq)$ and $(T, \psi)$ has valid $E$-base.
\end{theorem}
\else
\begin{theorem}[restate=THMsub, label=thm:sublattice]
For any standard closure space $(\U, \cl)$ with lattice $(\cs, \subseteq)$, there exists a standard closure space $(T, \psi)$ with lattice $(\cc{F}, \subseteq)$ such that $(\cs, \subseteq)$ is a sublattice of $(\cc{F}, \subseteq)$ and $(T, \psi)$ has valid $E$-base.
\end{theorem}
\fi

This result shows that lattices with valid $E$-base cannot be characterized by forbidden sublattices or universal sentences.

\paragraph{Organization of the paper.} In Section \ref{sec:preliminaries} we give necessary definitions and properties regarding closure spaces, lattices and implications.
In Section \ref{sec:E-base}, we introduce the $E$-base.
We provide examples to illustrate how it can fail being a valid IB of its associated closure space.
Moreover, we put the light on its relationships with other IBs, especially the canonical base.
Sections~\ref{sec:semidistributive}, \ref{sec:modular} and \ref{sec:geometric} are respectively devoted to semidistributive, modular and geometric lattices.
In Section~\ref{sec:sublattice}, we dive into the class of lattices with valid $E$-base.
We conclude the paper in Section \ref{sec:conclusion} with some questions for future work.

\section{Preliminaries} \label{sec:preliminaries}

All the objects we consider in this paper are finite.
We assume basic familiarity with lattices and closure systems and redirect the reader to \cite{gratzer2011lattice} for an introduction to the topic as well as for notions not defined here.
If $\U$ is a set, $\pow{\U}$ denotes its powerset.
Whenever a singleton set $\{X\}$ or $\{x\}$ is used, we may withdraw brackets and use $X$ and $x$ to simplify notations.
Sometimes, mostly in examples, we write a set as the concatenation of its elements, e.g.\ $abc$ instead of $\{a, b, c\}$.

\paragraph{Closure spaces.}
A \emph{closure operator} over $\U$ is a map $\cl \colon \pow{\U} \to \pow{\U}$ that satisfies the following properties for all $X, Y \subseteq \U$: (1) $X \subseteq \cl(X)$, (2) $\cl(\cl(X)) = X$ and (3) $\cl(X) \subseteq \cl(Y)$ if $X \subseteq Y$.
The pair $(\U, \cl)$ is a \emph{closure space}.
Let $\cs$ denote the fixpoints of $\cl$, i.e.\ $\cs = \{C \st C \subseteq \U, \cl(C) = C\}$.
The members of $\cs$ are called \emph{closed sets}.
The pair $(\U, \cs)$ is a \emph{closure system}, that is $\U \in \cs$ and $C_1 \cap C_2 \in \cs$ whenever $C_1, C_2 \in \cs$.
Note that $\cl(X) = \min_{\subseteq}\{C \st C \in \cs, X \subseteq C\} = \bigcap \{C \st C \in \cs, X \subseteq C\}$ for all $X \subseteq \U$.
The closure space $(\U, \cl)$ is \emph{standard} if $\cl(x) \setminus \{x\}$ is closed for every $x \in \U$.

\begin{remark}
Unless otherwise stated, we assume that all closure spaces are standard.
\end{remark}

Let $C \in \cs$.
A \emph{spanning set} of $C$ is any subset $X$ of $\U$ such that $\cl(X) = C$.
We say that $X$ \emph{spans} $C$. 
For $x \in \U$, a \emph{(non-trivial) generator} of $x$ is a subset $Y$ of $\U$ such that $x \notin Y$ but $x \in \cl(Y)$.
We say that $Y$ \emph{generates} $x$.
The family of inclusion-wise minimal generators of $x$ is denoted $\gen(x)$.

Let $(\U, \cl)$ be a closure space.
The \emph{binary part} of $(\U, \cl)$ is the closure space $(\U, \cl^b)$ defined for all $X \subseteq \U$ by:
\[ 
\cl^b(X) = \bigcup_{x \in X} \cl(x)
\]
Observe that any closed set of $\cs$ is also closed for $\cl^b$, or equivalently, $\cl^b(X) \subseteq \cl(X)$ for all $X \subseteq \U$.
Given $X, Y \subseteq \U$, $X$ is a \emph{$\cl^b$-refinement} of $Y$ if $\cl^b(X) \subseteq \cl^b(Y)$. 
If $\cl^b(X) \subset \cl^b(Y)$, $X$ is a \emph{proper} $\cl^b$-refinement of $Y$.
Let $C \in \cs$.
A $\cl^b$-minimal spanning set of $C$ is a spanning set $A$ of $C$ such that for any proper $\cl^b$-refinement $B$ of $A$, $B$ does not span $C$.
Since $B \subseteq A$ entails $\cl^b(B) \subseteq \cl^b(A)$, a $\cl^b$-minimal spanning set of $C$ is also a minimal spanning set of $C$.
Let $x \in \U$.
A \emph{(non-trivial) $\cl^b$-generator} $A$ of $x$ is a set $A$ such that $x \in \cl(A)$ but $x \notin \cl^b(A)$.
In particular, a $\cl^b$-generator $A$ of $x$ is a \emph{$D$-generator} of $x$ if it is $\cl^b$-minimal for this property, that is, $x \notin \cl(B)$ for any proper $\cl^b$-refinement $B$ of $A$.
The collection of $D$-generators of $x$ is denoted $\gen_D(x)$.
Observe that a $\cl^b$-generator of $x$ which is at the same time a $\cl^b$-minimal spanning set of its closed set needs not be a $D$-generator of $x$.
The \emph{$D$-relation} is a binary relation over $\U$ built upon $D$-generators as follows: $x D a$ if there exists a $D$-generator $A$ of $x$ that contains $a$.
A set $Q \subseteq \U$ is \emph{quasi-closed} w.r.t.~$\cl$ if for $X \subseteq Q$, $\cl(X) \subset \cl(Q)$ entails $\cl(X) \subseteq Q$.
Remark that if $\cl(Q) \neq \cl(x)$ for any $x \in \U$, then $Q$ is $\cl^b$-closed.
Indeed, for $x \in Q$, $\cl^b(x) = \cl(x) \subset \cl(Q)$ holds by assumption, so that $\cl^b(x) \subseteq Q$ follows.
If $P \subseteq \U$ is not closed and an inclusion-wise minimal quasi-closed set among the spanning sets of $\cl(P)$, it is \emph{pseudo-closed}.
The closed set $\cl(P)$ is \emph{essential}.

\paragraph{Closure lattices.}
Let $(\U, \cl)$ be a closure space.
The poset $(\cs, \subseteq)$ is a \emph{(closure) lattice} where $C_1 \mt C_2 = C_1 \cap C_2$ and $C_1 \jn C_2 = \cl(C_1 \cup C_2)$.
It is well-known that each lattice is isomorphic to a closure lattice.
In $(\cs, \subseteq)$, a \emph{predecessor} of $C \in \cs$ is a closed set $C'$ such that $C' \subset C$ and for every closed set $C''$, $C' \subset C'' \subseteq C$ means $C'' = C$.
The \emph{successors} of $C$ are defined dually.
We write $C' \prec C$ to denote that $C'$ is a predecessor of $C$ and $C' \succ C$ to denote that $C'$ is a successor of $C$.
We denote by $C_*$ the intersection of all the predecessors of $C$, i.e., $C_* = \bigcap_{C' \prec C} C'$.
Dually, $C^*$ denotes the supremum of all successors of $C$.
For $\cs' \subseteq \cs$, the \emph{ideal} $\idl \cs'$ of $\cs'$ in $(\cs, \subseteq)$ is given by $\idl \cs' = \{C \st C \in \cs, C \subseteq C' \text{ for some } C' \in \cs'\}$.
The \emph{filter} $\ftr \cs'$ of $\cs'$ is defined dually.
If $C \subseteq C'$ are two closed sets, the \emph{interval} between $C$ and $C'$, denoted by $[C, C']$, is the family of all closed sets that include $C$ and that are included in $C'$, that is, $[C, C'] = \{C'' : C'' \in \cs, C \subseteq C'' \subseteq C'\}$.

We define irreducible closed sets.
A closed set $J$ is \emph{join-irreducible}, \emph{ji} for short, if $J \neq \emptyset$ and for every $C_1, C_2 \in \cs$, $J = C_1 \jn C_2$ implies $J = C_1$ or $J = C_2$.
Equivalently, $J$ is ji iff it has a unique predecessor, being $J_*$.
Since $(\cs, \subseteq)$ is assumed standard, $J$ is ji if and only if $J = \cl(x)$ for some $x \in \U$.
As $\cl(x) \neq \cl(y)$ for $x \neq y$, $\cl$ is a bijection between $\U$ and the ji's of $(\cs, \subseteq)$.
Thus, for simplicity, we will identify $x$ with its closure $\cl(x)$ and say that $x$ is ji in $\cs$ with predecessor $x_*$.
A closed set $M$ is \emph{meet-irreducible}, \emph{mi} for short, if $M \neq \U$ and $M = C_1 \cap C_2$ implies $M = C_1$ or $M = C_2$ for any $C_1, C_2 \in \cs$.
Similarly to ji's, $M$ is mi if and only if it has a unique successor, being $M^*$.
An \emph{atom} of $(\cs, \subseteq)$ is a successor of $\emptyset$ ($\emptyset \in \cs$ since $(\U, \cl)$ is standard).
Dually, a \emph{coatom} is a predecessor of $\U$ in $(\cs, \subseteq)$.
Atoms are ji singletons, while coatoms are mi.

We move to arrow relations \cite{ganter2012formal}.
Let $J, M \in \cs$.
We write $J \upp M$ if $M$ is an inclusion-wise maximal closed set not containing $J$ as a subset, that is $J \nsubseteq M$ but $J \subseteq M'$ for every $M' \in \cs$ such that $M \subset M'$.
Note that $J \upp M$ implies that $M$ is mi.
Dually we write $J \downp M$ if $J \nsubseteq M$ but $J' \subseteq M$ for every $J' \subset J$ in $\cs$.
Similarly, $J \downp M$ entails that $J$ is join-irreducible. 
We write $J \dpp M$ if $J \upp M$ and $J \downp M$.
If $J$ is ji, we may write $x \upp M$, $x \downp M$ and $x \dpp M$ where $x$ is the unique element of $\U$ such that $J = \cl(x)$.
We turn our attention to prime and coprime elements \cite{markowsky1992primes}\footnote{The definition of primes and coprimes is switched w.r.t.~\cite{markowsky1992primes}.
In our context it is more natural to consider elements of $\U$ as primes rather than coprimes, since we will mostly use this aspect of primality.}.
A closed set $C$ is \emph{prime} in $(\cs, \subseteq)$ if for every $C_1, C_2 \in \cs$, $C \subseteq C_1 \jn C_2$ implies $C \subseteq C_1$ or $C \subseteq C_2$.
Prime elements must be ji, hence we say $x \in \U$ is prime if $\cl(x)$ is.
Equivalently, $x$ is prime if and only if there exists a unique mi $M$ such that $x \upp M$, in particular $x \dpp M$ must hold.
Moreover, $x$ is prime if and only if it admits only singleton minimal generators, i.e.~$\gen(x)$ contains only singleton sets.
Dually, a closed set $M$ is \emph{coprime} if for all $C_1, C_2 \in \cs$, $C_1 \cap C_2 \subseteq M$ implies $C_1 \subseteq M$ or $C_2 \subseteq M$. 
A coprime $M$ must be mi and there is a unique $J$ such that $M \downp J$, in particular it holds $J \dpp M$.
If $J$ is prime, the unique $M$ such that $J \dpp M$ is coprime.
Thus, $\dpp$ defines a bijection between prime and coprime closed sets.
The $D$-relation can be characterized in terms of arrow relations:  $x D a$ if $x \neq a$, $x \upp M \downp a$ for some mi closed set $M$.
We can also define the dual $D^*$ of the $D$-relation based on mi closed set instead, and where $M_1 D^* M_2$ holds if $M_1 \neq M_2$ and there exists $x \in \U$ such that $M_1 \downp x \upp M_2$.

\paragraph{Classes of lattices and closure spaces.}
We now introduce the classes of lattices we will refer to in this paper. 
Let $(\cs, \subseteq)$ be a (standard, closure) lattice.
If $C_1 \jn C_2 = C_1 \cup C_2$ for any $C_1, C_2 \in \cs$, $(\cs, \subseteq)$ is \emph{distributive}.
Equivalently, $(\cs, \subseteq)$ is distributive iff every element of $\U$ is prime.

We say that $(\cs, \subseteq)$ is \emph{join-semidistributive} if for all $C_1, C_2, C_3 \in \cs$,
$C_1 \jn C_2 = C_1 \jn C_3$ entails $C_1 \jn C_2 = C_1 \jn (C_2 \mt C_3)$.
\emph{Meet-semidistributivity} is defined dually by flipping $\mt$ and $\jn$ operations.
The lattice $(\cs, \subseteq)$ is \emph{semidistributive} if it is both meet and join-semidistributive.
We now recall useful properties regarding semidistributivity.

\begin{theorem}[Theorem 2.1, Corollary 2.2 in \cite{jonsson1962finite}, Corollary 1 in \cite{gorbunov1978canonical}] \label{thm:SDj-canonical}
The lattice of a (finite) closure space $(\U, \cl)$ is join-semidistributive if and only if each closed set $C$ has a unique $\cl^b$-minimal spanning set.
\end{theorem}

In the rest of the paper, we call this unique $\cl^b$-minimal spanning set the \emph{canonical spanning set} of $C$.

\begin{proposition}[Proposition 49 in \cite{adaricheva2014implicational}] \label{prop:SDj-UC}
In a closure space with join-semidistributive lattice, each essential closed set is spanned by a unique pseudo-closed set.
\end{proposition}

In particular, this pseudo-closed subsumes the canonical representation of its closure.

\begin{theorem}[Corollary 2.55, Theorem 2.56 in \cite{freese1995free}] \label{thm:SD-arrows}
Let $(\U, \cl)$ be a closure space with lattice $(\cs, \subseteq)$.
\begin{enumerate}[(1)]
    \item $(\cs, \subseteq)$ is join-semidistributive if and only if for every meet-irreducible $M$ there exists a unique $x \in \U$ such that $x \dpp M$.
    \item $(\cs, \subseteq)$ is meet-semidistributive if and only if for every $x \in \U$, there exists a unique meet-irreducible $M$ such that $x \dpp M$.
    \item $(\cs, \subseteq)$ is semidistributive if and only if $\dpp$ is bijection between $\U$ and meet-irreducible closed sets.
\end{enumerate}
\end{theorem}

\begin{theorem}[Theorem 1 in \cite{gorbunov1978canonical}] \label{thm:coatoms-SDj}
Let $(\U, \cl)$ be a (finite) closure space with join-semidistributive lattice $(\cs, \subseteq)$.
Let $M_1, \dots M_k$ be the coatoms of $(\cs, \subseteq)$.
Then, for all $1 \leq i \leq k$, $M_i$ is coprime and $A = \{a_1, \dots, a_k\}$, where $a_i$ is the unique (prime) $a_i$ such that $a_i \dpp M_i$, is the canonical spanning set of $\U$.
\end{theorem}

Observe that the last statement is proved in the more general case of complete strongly atomic lattices with the continuity condition.
If the $D$-relation associated to $(\cs, \subseteq)$ has no cycle, $(\cs, \subseteq)$ is \emph{lower-bounded}.
It is \emph{upper-bounded} if instead the $D^*$-relation is without cycles.
Lower-bounded lattices are join-semidistributive, and dually, upper-bounded lattices are meet-semidistributive.
A lattice which is both lower- and upper- bounded is \emph{bounded}.

The lattice $(\cs, \subseteq)$ is \emph{modular} if for every $C_1, C_2, C_3 \in \cs$ such that $C_1 \subseteq C_2$, $(C_2 \mt C_3) \jn C_1 = C_2 \mt (C_1 \jn C_2)$.
Equivalently, $(\cs, \subseteq)$ is modular if it satisfies both \emph{semimodular laws}: $(\U, \cs)$ is \emph{upper-semimodular} if for every $C_1, C_2 \in \cs$, $C_1 \cap C_2 \prec C_1$ entails $C_2 \prec C_1 \jn C_2$. 
\emph{Lower-semimodularity} is defined dually.
We will make use of the following characterization of essential and quasi-closed sets in modular lattice.
Let us recall that a lattice $(\cs, \subseteq)$ is a diamond if it consists in a bottom, atoms, and a top.
Notice that the first part of the statement was partially stated without proof in~\cite{herrmann1996polynomial} and later proved in~\cite{wild2000optimal}.

\begin{proposition} \cite[Proposition 4]{wild2000optimal} \label{prop:wild-modular}
A non-ji closed set $C$ is essential iff $[C_*, C]$ is isomorphic to a diamond with at least $3$ ji elements.
Moreover, if $C$ has predecessors $C_1, \dots, C_m$, then the quasi-closed sets spanning $C$ are precisely the sets
\[ \bigcup_{i \in I} C_i \quad \text{ with $I \subseteq \{1, \dots, m\}$ and $1 < \card{I} < m$}. \] 
\end{proposition}

The lattice $(\cs, \subseteq)$ is \emph{geometric} if it is atomistic---i.e., every join-irreducible is an atom---and upper-semimodular.
If $(\cs, \subseteq)$ is geometric, the associated closure space $(\U, \cl)$ has the \emph{exchange property}: for every $C \in \cs$ and every $x, y \notin C$, $x \in \cl(C \cup y)$ implies $y \in \cl(C \cup x)$.
There is a well-known one-to-one correspondence between (finite) geometric lattices and matroids.
We briefly recall the link between the two structure, and refer to \cite{oxley2006matroid} for a detailed introduction to the topic.
If $(\cs, \subseteq)$ is a geometric lattice over $\U$, then, $\cc{M} = (\U, \cc{I})$ where $\cc{I} = \{I \st I \subseteq \U \text{ is a minimal spanning set of } \cl(I)\}$ is the matroid associated to $(\cs, \subseteq)$.
The bases of $\cc{M}$ are the inclusion-wise maximal sets of $\cc{I}$, i.e., the minimal spanning sets of $\U$ in $(\cs, \subseteq)$.
We denote by $\cc{B}$ the family of bases of $\cc{M}$.
It is well-known that all bases of $\cc{M}$ have the same size and that they are subject to the base exchange axiom: for every distinct $B_1, B_2 \in \cc{B}$ and any $x \in B_1 \setminus B_2$, there exists $y \in B_2 \setminus B_1$ such that $(B_1 \setminus \{x\}) \cup \{y\} \in \cc{B}$.
A circuit is an inclusion-wise minimal set $K \subseteq \U$ such that $K \nsubseteq B$ for any base $B$.
Let us call $\cc{K}$ the family of circuits of $\cc{M}$.
Circuits are closely connected to minimal generators.
Namely, for $x \in \U$, the minimal generators of $x$ are precisely the sets $K \setminus x$ where $K$ is a circuit containing $x$, that is, $\gen(x) = \{K \setminus x : K \in \cc{K}, x \in K\}$.
Note that since $(\cs, \subseteq)$ is atomistic, $\gen(x) = \gen_D(x)$ for any $x \in \U$.

Finally, $(\cs, \subseteq)$ is \emph{meet-distributive} if it is lower-semimodular and join-semidistributive, and \emph{join-distributive} if it enjoys the dual laws.
Meet-distributive lattices are also known as convex geometries.
In this case the corresponding closure operator $\cl$ has the \emph{anti-exchange property}, dual to the exchange property: for every $C \in \cs$ and every distinct $x, y \notin C$, $x \in \cl(C \cup y)$ implies $y \notin \cl(C \cup x)$.

\paragraph{Implicational bases (IBs).} 
We refer to \cite{bertet2018lattices,wild2017joy} for a more detailed introduction to the topic.
An \emph{implication} over $\U$ is a statement $A \imp X$, with $A, X \subseteq \U$.
In $A \imp X$, $A$ is the \emph{premise} and $X$ the $\emph{conclusion}$.
An implication $A \imp x$, $x \in \U$, is a \emph{(right) unit implication}.
If both $A$ and $X$ are singletons, the implication is \emph{binary}.
An \emph{implicational base} (IB) is a pair $(\U, \is)$ where $\is$ is a set of implications over $\U$.
An IB is unit (resp.~binary) if all its implications are unit (resp.~binary).
The \emph{aggregation} or \emph{aggregated form} of a unit IB $(\U, \is)$ is obtained by replacing all implications with the same premise by a single implication with unified conclusions.

An IB $(S, \is)$ induces a closure operator $\is(\cdot)$, hence a closure space,  where a subset $C$ of $\U$ is closed if $A \subseteq C$ implies $X \subseteq C$ for all $A \imp X \in \is$.
Usually, each closure space $(\U, \cl)$---each lattice $(\cs, \subseteq)$---can be represented by several \emph{equivalent} IBs.
An implication $A \imp X$ \emph{holds} in $(\U, \cl)$ if for every $C \in \cs$, $A \subseteq C$ implies $X \subseteq C$.
An implication $A \imp X$ holds in $(\U, \cl)$ if and only if $X \subseteq \cl(A)$.
If $(\U, \is)$ is an IB with closure space $(\U, \cl)$, we say that an implication $A \imp X$ \emph{holds} in, or \emph{follows from} $(\U, \is)$ if it holds in $(\U, \cl)$.
Thus, two IBs $(\U, \is_1)$, $(\U, \is_2)$ are equivalent if every implication of $\is_1$ follows from $(\U, \is_2)$ and vice versa.
The binary part of a closure space $(\U, \cl)$ can be represented by the binary IB $(\U, \is^b)$ gathering all valid (non-trivial) binary implications, that is $\is^b = \{a \imp x \st x \in \cl(a), x \neq a\}$.
An IB $(\U, \is)$ is \emph{valid} for $(\U, \cl)$ if the closure space associated to $(\U, \is)$ is indeed $(\U, \cl)$.
Among the possible IBs of $(\U, \cl)$, and besides the $E$-base, we will mention three:
\begin{itemize}
\item the \emph{canonical} or \emph{Duquenne-Guigues base} $(\U, \is_{DG})$ \cite{guigues1986familles}.
It is defined from pseudo-closed sets as follows: 
\[ \is_{DG} = \{P \imp \cl(P) \setminus P \st P \text{ is pseudo-closed}\}. \]
This IB is \emph{minimum}, it has the least number of implications among all other IBs.
The following theorem is well-known and will be useful in the next sections.
\begin{theorem}[e.g.~Theorem 5 in \cite{wild1994theory}] \label{thm:canonical}
Let $(\U, \cl)$ be a closure space and let $(\U, \is)$ be an IB for $(\U, \cl)$. 
Then, for any pseudo-closed set $P$ of $(\U, \cl)$, $\is$ contains an implication $A \imp B$ such that $A \subseteq P$ and $\cl(A) = \cl(P)$.
\end{theorem}

\item the \emph{canonical direct base} $(\U, \is_{cd})$ \cite{bertet2010multiple}.
It relies on minimal generators of each element of $\U$.
More precisely, 
\[ \is_{cd} = \{A \imp x \st x \in \U, A \in \gen(x)\}.\]

\item the \emph{$D$-base} $(\U, \is_D)$ \cite{adaricheva2013ordered}.
It is defined from $D$-generators and binary implications representing the binary part of $(\U, \cl)$: 
\[ \is_D = \is^b \cup \{A \imp x \st x \in \U, A \in \gen_D(x)\}. \]
Remark that $\is_D \subseteq \is_{cd}$.
\end{itemize}

Note that for any $C \in \cs$, the restriction of $\is_{DG}$ (resp.\! $\is_{cd}$, $\is_D$) to the implications whose premise and conclusion are subsets of $C$ constitutes the canonical base (resp. canonical direct, $D$-base) of the closure system $(C, \idl C)$.

\section{The \mtt{$E$}{E}-base and some of its properties} \label{sec:E-base}

In this section, we first introduce the $E$-base in the line of seminal works on these objects \cite{adaricheva2013ordered,freese1995free}.
Then, we exhibit further properties of the $E$-base and connect it with other IBs, especially the canonical base.

Let us fix a (standard) closure space $(\U, \cl)$ with lattice $(\cs, \subseteq)$.
We start with the definitions of $E$-generators and $E$-base.
These definitions translate those of \cite{adaricheva2013ordered,freese1995free} within our terms.

\begin{definition}[$E$-generator] \label{def:E-generator}
Let $(\U, \cl)$ be a closure space and let $x \in \U$.
A subset $A$ of $\U$ is a \emph{$E$-generator} of $x$ if $A$ is a $D$-generator of $x$ and $\cl(A)$ is inclusion-wise minimal in $\{\cl(B) \st B \in \gen_D(x)\}$.
We denote by $\gen_E(x)$ the family of $E$-generators of $x$. We have
\[
\gen_E(x) = \{A \st A \in \gen_D(x) \text{ and } \cl(A) \in \min_{\subseteq}\{\cl(B) \st B \in \gen_D(x)\}\}.
\]
\end{definition}

In general $\gen_E(x)$ may be empty, much as $\gen_D(x)$.
In fact, this will happen precisely if $x$ is prime in $(\cs, \subseteq)$.
Moreover, as pointed out in \cite{wild2017joy}, $E$-generators share common points with critical generators of convex geometries \cite{korte2012greedoids}.
Convex geometries form a well-known class of closure spaces.
For a convex geometry, a critical generator of some $x$ is a minimal generator $A$ of $x$ whose closure is minimal among the closures of other minimal generators of $x$.
Thus, $E$-generators and critical circuits both asks for minimality in the closure lattice.
Yet, the two objects differ in that $E$-generators rely on $\cl^b$ while critical generators do not.

\begin{definition}[$E$-base] \label{def:E-base}
The \emph{$E$-base} of a closure space $(\U, \cl)$ is the IB $(\U, \is_E)$ where
\[ 
\is_E = \is^b \cup \{A \imp x \st x \in \U, A \in \gen_E(x)\}.
\]
\end{definition}

By definition, the $E$-base is a subset of the $D$-base.
Henceforth, it is also a subset of the canonical direct base.
Intuitively, the $E$-base is obtained from the $D$-base by ruling out implications spanning closed sets ``too high'' in the closure lattice with respect to a fixed element $x$.
More formally, one obtains the $E$-base from the $D$-base by repeatedly applying the following operation: if $A_1 \imp x, A_2 \imp x$ lie in the non-binary part of $\is_D$ and $\cl(A_1) \subset \cl(A_2)$, then remove $A_2 \imp x$.
\begin{example}[see Figure 5.5, p.~111 in \cite{freese1995free}] \label{ex:carpet}
Let $\U = \{a, b, c, d, e, f, g\}$ and consider the closure space $(\U, \cl)$ associated to the semidistributive lattice of Figure \ref{fig:carpet}.
We construct the $E$-base by identifying, for each $x$, $\gen(x)$, $\gen_D(x)$ and finally $\gen_E(x)$. 
\begin{itemize}
    \item $\gen(a) = \{c, e, f, g\}$. 
    Since $\gen_D(x)$ takes into account only non-singleton minimal generators, we have $\gen_D(a) = \emptyset$ and hence $\gen_E(a) = \emptyset$ as well ($a$ is prime). 
    
    \item $\gen(b) = \{d\}$.
    As for $b$, we have $\gen_E(b) = \emptyset$ ($b$ is prime too).
    
    \item $\gen(c) = \{\agset{ad}, e, f, g\}$.
    Since $\agset{ad}$ is the unique non-singleton minimal generator of $c$, it must be at the same time a $D$-generator and a $E$-generator for it is $\cl^b$-minimal and minimal with respect to $\cl$.
    Thus, $\gen_E(c) = \{\agset{ad}\}$.
    
    \item $\gen(d) = \{\agset{be}, \agset{bg}\}$.
    We have $\cl^b(\agset{bg}) = \agset{abcfg}$ and $\cl^b(\agset{be}) = \agset{abce}$.
    Thus both are $\cl^b$-minimal and $\gen_D(d) = \{\agset{be}, \agset{bg}\}$.
    Now to identify $\gen_E(d)$ we need to compare the closures of the $D$-generators of $d$.
    We have $\cl(\agset{bg}) = \agset{abcdfg} \subset \U = \cl(\agset{be})$. 
    Ruling out $be$, we get $\gen_E(d) = \{\agset{bg}\}$.
    
    \item $\gen(e) = \emptyset$ so that $\gen_E(e) = \emptyset$. 
    
    \item $\gen(f) = \{\agset{bc}, \agset{be}, \agset{cd}, \agset{ad}, \agset{de}, g\}$.
    Among them, we have $\cl^b(\agset{bc}) = \agset{abc} \subset \agset{abce} = \cl^b(\agset{be})$ and $\cl^b(\agset{ad}) = \agset{abd} \subset \agset{abcd} = \cl^b(\agset{cd}) \subset \agset{abcde} = \cl^b(\agset{de})$.
    Keeping only the $\cl^b$-minimal generators leads to $\gen_D(f) = \{\agset{ad}, \agset{bc}\}$.
    We have $\cl(\agset{bc}) = \agset{abcf} \subset \agset{abcdf} = \cl(\agset{ad})$ which leads to withdraw $\agset{ad}$.
    We thus have $\gen_E(f) = \{\agset{bc}\}$.
    
    \item $\gen(g) = \{\agset{be}, \agset{de}, \agset{ef}\}$.
    We have $\cl^b(\agset{be}) = \agset{abce} \subset \agset{abcde} =\cl^b(\agset{de})$. 
    On the other hand $\cl^b(\agset{ef}) = \agset{acef}$.
    We obtain $\gen_D(g) = \{\agset{be}, \agset{ef}\}$.
    Now $\cl(\agset{ef}) = \agset{acefg} \subset \agset{abcdefg} = \cl(\agset{be})$.
    Henceforth, $\gen_E(g) = \{\agset{ef}\}$.
    
\end{itemize}
Putting all the parts together, we can build the $E$-base $(\U, \is_E)$ of $(\U, \cl)$.
We separate the implications into groups as follows: first we present (aggregated) binary implications altogether, and then we split the implications according to the closure of their premises.
We will frequently use this decomposition all along the paper as it will ease the comparison between different IBs.
\begin{align*}
\is_E & = \{\agset{g \imp acf}, \agset{f \imp ac}, \agset{e \imp ac}, \agset{d \imp b}, \agset{c \imp a} \} & (\text{binary}) \\ 
& \cup \{\agset{bc \imp f}\} & (\agset{abcf}) \\ 
& \cup \{\agset{ad \imp c}\} & (\agset{abcdf}) \\ 
& \cup \{\agset{ef \imp g}\} & (\agset{acefg}) \\
& \cup \{\agset{bg \imp d}\} & (\agset{abcdfg}) 
\end{align*}

\begin{figure}[ht!]
    \centering
    \includegraphics[scale=\FIGcarpet]{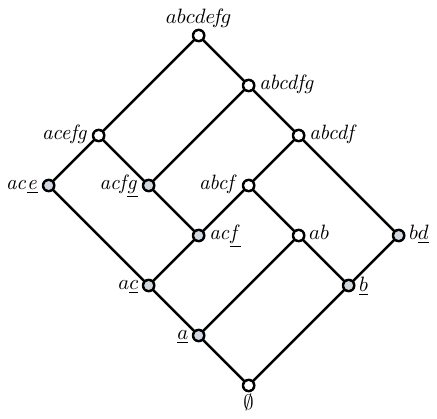}
    \caption{The closure lattice of Example \ref{ex:carpet}. A shaded dot indicates a join-irreducible closed set.
    The element of which it is the closure is underlined in the label, e.g., $\cl(f) = \agset{acf}$.
    The lattice is semidistributive.}
    \label{fig:carpet}
\end{figure}

\end{example}
Unlike the $D$-base though, and as mentioned in the introduction, the $E$-base of a closure space does not always constitute a valid IB.
We give below three increasingly demanding examples that illustrate how the $E$-base can fail to describe a closure space by comparing it to the canonical base. 
Indeed, the $E$-base $(\U, \is_E)$ of $(\U, \cl)$ is not valid if and only if it its not equivalent to the canonical base $(\U, \is_{DG})$ of $(\U, \cl)$, that is, if there exists some implication $P \imp C \setminus P$ in $\is_{DG}$ that does not follow from $\is_E$.
This prevents some spanning sets of the essential set $C$, $P$ in particular, to be correctly mapped on $C$ by the $E$-base.
In the rest of the paper, we will say that such pseudo-closed and essential sets are \emph{faulty} (w.r.t.\ the $E$-base).
An essential set is faulty if and only if it has a faulty pseudo-closed spanning set, and $(\U, \is_E)$ is valid if and only if there is no faulty essential sets in $(\U, \cl)$.
Remark that, by definition of the $E$-base, a ji essential set cannot be faulty.
Hence, when dealing with faulty or non-faulty essential set, we always assume that the corresponding essential set is non-ji.

\begin{example}[Example 27 in \cite{adaricheva2013ordered}] \label{ex:leaf}
Let $\U = \{a, b, c, d\}$ and consider the closure space $(\U, \cl)$ whose closure lattice is given in Figure \ref{fig:leaf}.
It is a convex geometry, and hence join-semidistributive.
However, is it not meet-semidistributive.
We have:

\begin{center}
\begin{minipage}{0.45\textwidth}
\begin{align*}
\is_{DG} & = \{\agset{ac \imp b}\} & (\agset{abc}) \\
& \cup \{\agset{bd \imp c}\} & (\agset{bcd}) \\
& \cup \{\agset{ad \imp bc}\} & (\agset{abcd})
\end{align*}
\end{minipage}
\begin{minipage}{0.45\textwidth}
\begin{align*}
\is_{E} & = \{\agset{ac \imp b}\} & (\agset{abc}) \\
& \cup \{\agset{bd \imp c}\} & (\agset{bcd}) \\
\end{align*}
\end{minipage}
\end{center}

Note that $\is_{DG}$ coincides with the (aggregated) $D$-base.
The implication $\agset{ad \imp bc}$ does not belong to the (aggregated) $E$-base of $(\U, \cl)$.
Indeed, for $\agset{ad \imp c}$, we have $\agset{bd \imp c}$ with $\cl(\agset{bd}) = \agset{bcd} \subset \agset{abcd} = \cl(\agset{ad})$.
Similarly for $\agset{ad \imp b}$, we have $\agset{ac \imp b}$ in $\is$ with $\cl(\agset{ac}) = \agset{abc} \subset \agset{abcd} = \cl(\agset{abcd})$.
It follows that no $E$-generator, hence no premise of the $E$-base, spans the essential set $\agset{abcd}$.
This makes the pseudo-closed set $ad$ and the essential set $\agset{abcd}$ faulty and hence the $E$-base non valid.
\begin{figure}[ht!]
    \centering
    \includegraphics[scale=\FIGleaf]{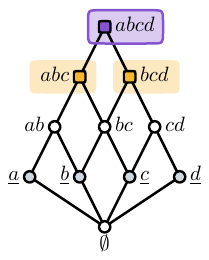}
    
    \includegraphics[page=5]{figures/captions.pdf}
    \caption{The convex geometry of Example \ref{ex:leaf}. 
    Square dots are the essential sets: $\agset{abc}$, $\agset{bcd}$ and $\agset{abcd}$.
    Among them, $\agset{abc}$ and $\agset{bcd}$ (highlighted in yellow) are non-faulty.
    On the other hand, $\agset{abcd}$ (circled in purple) is faulty since it is not spanned by any $E$-generator.}
    \label{fig:leaf}
\end{figure}
\end{example}

The previous example suggests that the $E$-base of a closure space may not be valid because some essential set is not spanned by any $E$-generator.
Yet, the next example shows that even if each essential set is indeed spanned by some $E$-generator, the $E$-base may not be valid.

\begin{example} \label{ex:S7-gluing}
Consider the closure space $(\U, \cl)$ corresponding to the lattice of Figure \ref{fig:S7-gluing}.
Below, we give the associated canonical and $E$-base:

\begin{center}
\begin{minipage}{.45\textwidth}
\begin{align*}
\is_{DG} & =  \{\agset{d \imp c}, \agset{e \imp c}\} & (\text{binary})\\ 
& \cup \{\agset{ac \imp b}, \agset{bc \imp a}\} & (abc) \\ 
& \cup \{\agset{cde \imp ab}, \agset{abce \imp d}\} & (abcde) \\
\end{align*}
\end{minipage}
\begin{minipage}{.45\textwidth}
\begin{align*}
\is_E & = \{\agset{d \imp c}, \agset{e \imp c}\} & \hfill (\text{binary}) \\ 
& \cup \{\agset{ac \imp b},\agset{bc \imp a}\} & \hfill (abc) \\ 
& \cup \{\agset{ae \imp d}, \agset{be \imp d}\} & \hfill (abcde) \\
\end{align*}
\end{minipage}
\end{center}

The essential sets are $\agset{cd}$, $\agset{ce}$, $\agset{abc}$, and $\agset{abcde}$.
They are all spanned by some $E$-generator. 
The pseudo-closed set $\agset{cde}$ does not include any $E$-generator as shown in the box of Figure \ref{fig:S7-gluing}.
Hence, $(\U, \is_E)$ is not valid due to Theorem~\ref{thm:canonical}.
In fact, $\agset{cde}$ is even closed w.r.t.~$(\U, \is_E)$ , which makes $\agset{abcde}$ faulty.

\begin{figure}[ht!]
\centering
\includegraphics[width=\FIGgluing\textwidth]{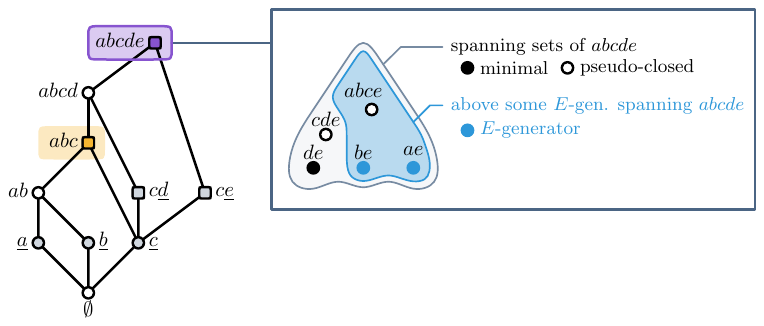}
\includegraphics[width=\textwidth, page=2]{figures/captions.pdf}
\caption{The closure lattice of Example \ref{ex:S7-gluing}.  
The essential sets are $\agset{abc}$, $\agset{cd}$, $\agset{ce}$ and $\agset{abcde}$ (square nodes).
All are spanned by some $E$-generator.  
Yet, the pseudo-closed set $\agset{cde}$ spanning $\agset{abcde}$ does not include any $E$-generator, while $\agset{abce}$ includes $\agset{ae}$ and $\agset{be}$.
This is shown in the box on the right picturing the spanning sets of $\agset{abcde}$.
}
\label{fig:S7-gluing}
\end{figure}

\end{example}

In the previous example, the $E$-base fails to describe the closure space because one of the pseudo-closed sets is not subsumed by any $E$-generator spanning the same essential set.
However, even if the $E$-base satisfies this condition, it does not need to be valid.

\begin{example} \label{ex:S7-relaxed}
Consider the closure space $(\U, \cl)$ associated to the lattice of Figure \ref{fig:S7-relaxed}.
We have:

\begin{center}
\begin{minipage}{0.47\textwidth}
\begin{align*}
\is_{DG} & = \{\agset{d \imp c}, \agset{f \imp ce}, \agset{e \imp c}\} & (\text{binary}) \\
& \cup \{\agset{ac \imp bd}, \agset{bc \imp ad} \} & (\agset{abcd}) \\ 
& \cup \{\agset{cde \imp abf} \} & (\agset{abcdef})
\end{align*}
\end{minipage}
\begin{minipage}{0.47\textwidth}
\begin{align*}
\is_{E} & = \{\agset{d \imp c}, \agset{f \imp ce}, \agset{e \imp c}\} & (\text{binary}) \\
& \cup \{\agset{ac \imp bd}, \agset{bc \imp ad}\} & (\agset{abcd}) \\ 
& \cup \{\agset{ae \imp f}, \agset{be \imp f}, \agset{de \imp f}\} & (\agset{abcdef})
\end{align*}
\end{minipage}
\end{center}

In this case, the pseudo-closed set $cde$ indeed includes the $E$-generator $de$ subsuming the same essential set $\agset{abcdef}$.
However, the closure of $de$ in $\is_E$ does not reach $a$ and $b$, thus making $\agset{cde}$ and $\agset{abcdef}$ faulty.
\begin{figure}[ht!]
    \centering
    \includegraphics[scale=\FIGrelaxed]{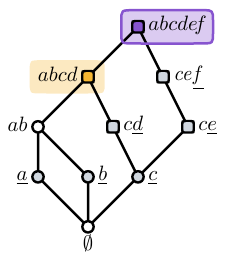}
    \includegraphics[width=\textwidth, page=2]{figures/captions.pdf}
    \caption{The lattice of Example \ref{ex:S7-relaxed} where $\agset{abcdef}$ is faulty and $\agset{abcd}$ is not.
    Even though the pseudo-closed set $\agset{cde}$ is subsumed by the $E$-generator $\agset{de}$, it remains faulty.}
    \label{fig:S7-relaxed}
\end{figure}
\end{example}

Understanding what closure spaces have valid $E$-base is an intriguing problem, and constitutes the main motivation for our contribution.
Before investigating this question, we give a characterization of $E$-generators in Lemma \ref{lem:E-generator}.
It comes from the observation that if $A \in \gen_E(x)$, any closed set $C \subset \cl(A)$ generating $x$ can only contain singleton minimal generators of $x$.
This makes $x$ prime in the lattice $(\idl C, \subseteq)$.

\begin{lemma} \label{lem:E-generator}
A subset $A$ of $\U$ is a $E$-generator of $x$ if and only if the following conditions hold:
\begin{enumerate}[(1)]
    \item $A$ is a $\cl^b$-minimal spanning set of $\cl(A)$
    \item $A$ is a non-trivial $\cl^b$-generator of $x$
    \item for every $C \in \cs$ such that $C \subset \cl(A)$, $x \in C$ implies that $x$ is prime in $(\idl C, \subseteq)$.
\end{enumerate}
\end{lemma}

\ifx\arxiv\undefined
\else
\begin{proof}
We start with the only if part.
Let $A \in \gen_E(x)$.
By definition, $A \in \gen_D(x)$ so that conditions (1) and (2) readily holds.
We prove that condition (3) is satisfied.
Let $C \in \cs$ be such that $C \subset \cl(A)$ and $x \in C$.
Assume for contradiction $x$ is not prime in $(\idl C, \subseteq)$.
Then, there exists two closed sets $C_1, C_2 \subset C$ such that $x \notin C_1, C_2$ but $x \in C_1 \jn C_2 = \cl(C_1 \cup C_2)$. 
It follows that $C_1 \cup C_2$ contains a minimal generator $B$ of $x$.
As $C_1, C_2 \in \cs$, $C_1 \cup C_2$ is $\cl^b$-closed.
Henceforth, $\card{B} \geq 2$ must hold.
Moreover $B$ can be chosen $\cl^b$-minimal, in which case $B$ will be a $D$-generator of $x$.
But then, $\cl(B) \subseteq C \subset \cl(A)$, which contradicts $A$ being a $E$-generator of $x$.
Therefore, all conditions (1), (2) and (3) are satisfied.

We move to the if part.
Assume all three conditions hold.
We prove that $A \in \gen_E(x)$.
First, let us show that $A \in \gen_D(x)$.
Let $B$ be any proper $\cl^b$-refinement of $A$.
By condition (1), $\cl(B) \subset \cl(A)$ must hold.
From condition (2), $x \in \cl(A)$ but $x \notin \cl^b(A)$.
As $B$ $\cl^b$-refines $A$, we deduce $x \notin \cl^b(B)$.
On the other hand, $x \in \cl(B)$ would imply that $x$ is prime in $(\idl \cl(B), \subseteq)$ by condition (3).
Since $x \notin B$, this would lead to $x \in \cl^b(B)$, a contradiction.
We deduce that $x \notin \cl(B)$.
This concludes the proof that $A \in \gen_D(x)$.
It remains to show that $A \in \gen_E(x)$.
Let $B$ be any $D$-generator of $x$ distinct from $A$.
Since having $\cl(B) \subset \cl(A)$ would contradict condition (3), we conclude that $\cl(B) \not\subset \cl(A)$.
In other words, $\cl(A)$ is minimal in $\{\cl(B) \st B \in \gen_D(x)\}$ and $A \in \gen_E(x)$ holds as required.
\ifx\arxiv\undefined
\qed
\fi
\end{proof}
\fi

\begin{remark}
Recall from the preliminaries that conditions (1) and (2) in Lemma \ref{lem:pseudo-E} are not equivalent to be a $D$-generator.
While a $D$-generator $A$ of $x$ satisfies these conditions it also needs to be $\cl^b$-minimal in the sense that no proper $\cl^b$-refinement of $A$ generates $x$ too. 
\end{remark}

Condition (3) in the last lemma is equivalent to the condition (3') below:
\begin{itemize}
    \item[(3')] for every $C \in \cs$ such that $C \prec \cl(A)$, $x \in C$ implies that $x$ is prime in $(\idl C, \subseteq)$.
\end{itemize}

In words, Lemma \ref{lem:E-generator} describes by means of $\cl^b$-minimal spanning sets the inclusion-wise minimal closed sets above which $x$ stops being prime, i.e., when it starts admitting non-singleton minimal generators.
This is formally stated in the subsequent corollary.

\begin{corollary} \label{cor:E-spanned}
The family $\{C \st C \in \cs \text{ and } x \text{ is not prime in } (\idl C, \subseteq)\}$ is a filter of $(\cs, \subseteq)$ whose inclusion-minimal members are exactly the closed sets of the $E$-generators of $x$:
\[ 
\{\cl(A) \st A \in \gen_E(x)\} = \min_{\subseteq}\{C \st C \in \cs \text{ and } x \text{ is not prime in } (\idl C, \subseteq)\}.
\]
\end{corollary}

Lemma \ref{lem:E-generator} also stresses on the fact that any $\cl^b$-minimal spanning set $A$ of the closure of some $E$-generator of $x$ is also a $E$-generator of $x$, as long as $x \notin \cl^b(A)$.

\begin{corollary} \label{cor:E-min-span-set}
If $C \in \{\cl(B) \st B \in \gen_E(x)\}$, then any $\cl^b$-minimal spanning set $A$ of $C$ satisfying $x \notin \cl^b(A)$ is a $E$-generator of $x$.
\end{corollary}

\ifx\arxiv\undefined
\else
\begin{proof}
Since $x \in \cl(A) = C$, $A$ and $x$ readily satisfy conditions (1) and (2) of Lemma \ref{lem:E-generator} by assumption.
Condition (3) follows from $\cl(A) = \cl(B)$ and $B \in \gen_E(x)$.
\ifx\arxiv\undefined
\qed
\fi
\end{proof}
\fi

\begin{figure}[ht!]
    \centering
    \includegraphics[width=\textwidth]{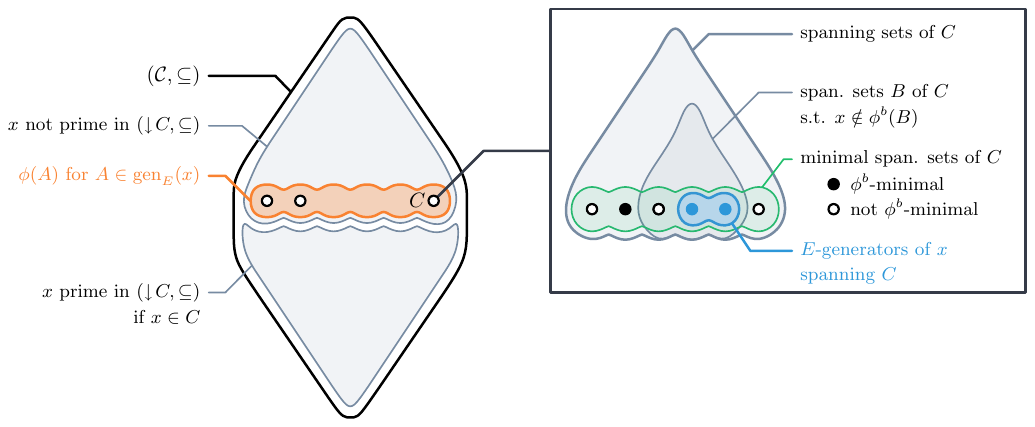}
    \caption{The $E$-generators of $x$ in a closure space $(\U, \cl)$. On the left, the lattice $(\cs, \subseteq)$ associated with $(\U, \cl)$ is partitioned according to Corollary \ref{cor:E-spanned}, i.e., according to the filter of closed sets above which $x$ is not prime. The minimal members of this filter are the closure of the $E$-generators of $x$.
    The box on the right illustrates Corollary \ref{cor:E-min-span-set}: within the spanning sets of $C = \cl(A)$ for some $A \in \gen_E(x)$, the $E$-generators of $x$ are the $\cl^b$-minimal spanning sets $B$ of $C$ s.t.\ $x \notin \cl^b(B)$.}
    \label{fig:E-gen-structure}
\end{figure}

The situation is pictured in Figure \ref{fig:E-gen-structure}.
Following Lemma \ref{lem:E-generator} and Corollary \ref{cor:E-spanned}, for a closed set $C$ and $x \in C$, we say that $x$ is \emph{almost prime} in $(\idl C, \subseteq)$ if $x$ is not prime in $(\idl C, \subseteq)$ but prime in $(\idl C_*, \subseteq)$ for every $C_* \prec C$ such that $x \in C_*$.
We thus obtain a new definition of the $E$-base of a closure space.

\begin{definition}[$E$-base] \label{def:E-base-prime}
Let $(\U, \cl)$ be a closure space with lattice $(\cs, \subseteq)$.
The $E$-base of $(\U, \cl)$ is the IB $(\U, \is_E)$ where 
\[
\begin{split}
\is_E = & \is^b \cup \{A \imp x \st x \in \U,  A \text{ is a $\cl^b$-minimal spanning set of $\cl(A)$} \\
        & \quad\quad \text{s.t.\ } x \notin \cl^b(A) \text{ and $x$ is almost prime in } (\idl \cl(A), \subseteq)\}.
\end{split}
\] 
\end{definition}

Then, we show that $E$-generators and their closure relate to any IB representing the closure space.
To do so, we highlight the relationships between $E$-generators and quasi-closed sets.
The next lemma can be seen as the analogue for $E$-generators in any closure space of the Expansion 16 of Wild \cite{wild2017joy} dedicated to critical generators in convex geometries.

\begin{lemma} \label{lem:E-QC}
Let $C \in \cs$, $x \in C$, and $Q = C \setminus \{y \st x \in \cl^b(y)\}$.
Then $Q$ is a quasi-closed set spanning $C$ if and only if $C = \cl(A)$ for some $A \in \gen_E(x)$.
\end{lemma}

\ifx\arxiv\undefined
\else
\begin{proof}
We start with the only if part.
Assume $Q = C \setminus \{y \st x \in \cl^b(y)\}$ spans $C$ and is quasi-closed.
Note that since $Q$ spans $C$ and $(\U, \cl)$ is standard, $C \neq \cl(z)$ for any $z \in C$.
We show that $C = \cl(A)$ for some $A \in \gen_E(x)$. 
Let $A$ be a $\cl^b$-minimal spanning set of $C$ included in $Q$.
Since $Q$ is quasi-closed, $\cl^b(z) \subseteq Q$ for each $z \in Q$ as $C \neq \cl(z)$.
Hence, such a $A$ must exist.
We prove $A$ is a $E$-generator of $x$.
Conditions (1) and (2) of Lemma \ref{lem:E-generator} are already satisfied by definition of $A$, $Q$ and $C$.
We argue that condition (3) also holds.
Let $C' \in \cs$ be such that $x \in C'$ and $C' \subset C$.
To prove that $x$ is prime in $(\idl C', \subseteq)$, we need only prove that $C'$ does not include any  non-singleton minimal generator of $x$.
This amounts to show that $C' \setminus \{y \st x \in \cl^b(y)\}$ contains no minimal generator $x$.
We have $C' \setminus \{y \st x \in \cl^b(y)\} \subseteq C \setminus \{y \st x \in \cl^b(y)\} = Q$.
Moreover $C' \subset C$ so that $\cl(C' \setminus \{y \st x \in \cl^b(y)\}) \subseteq C' \subset C$.
Since $Q$ is quasi-closed, we deduce $\cl(C' \setminus \{y \st x \in \cl^b(y)\}) \subseteq Q$.
As $x \notin Q$, we obtain that $C'$ contains no non-singleton minimal generator of $x$.
Therefore $x$ is indeed prime in $(\idl C', \subseteq)$ for any closed set $C' \subset C$ containing $x$.
Condition (3) of Lemma \ref{lem:E-generator} is satisfied, and $A \in \gen_E(x)$ holds as required. 
This ends this part of the proof.

We move to the if part.
Assume $C = \cl(A)$ for some $E$-generator $A$ of $x$.
Let $Q = C \setminus \{y \st x \in \cl^b(y)\}$.
We argue that $Q$ is a quasi-closed set spanning $C$.
Since $A \subseteq Q$, $\cl(Q) = C$ readily holds.
We show that $Q$ is quasi-closed.
Let $X \subseteq Q$ such that $\cl(X) \subset C$. 
We prove that $\cl(X) \subseteq Q$.
Assume for contradiction $\cl(X) \nsubseteq Q$.
Then, there exists $y \in C \setminus Q$ such that $y \in \cl(X)$.
Since $x \in \cl^b(y) = \cl(y)$, we have $x \in \cl(X)$.
Henceforth, $x \in \cl(X)$ and $X$ includes a minimal generator $B$ of $x$.
By definition $B$ is not a singleton.
We deduce that $x$ is not prime in $(\idl \cl(X), \subseteq)$.
Thus condition (3) of Lemma \ref{lem:E-generator} fails, which contradicts $A$ being a $E$-generator of $x$.
Hence $\cl(X) \subseteq Q$ holds, and $Q$ is quasi-closed.
\ifx\arxiv\undefined
\qed
\fi
\end{proof}
\fi

As a consequence, the closure of any $E$-generator must be essential.
On the other hand, not all essential need be spanned by some $E$-generator, as shown by Example \ref{ex:leaf}.
We can now give a theorem that relates $E$-generators, hence the $E$-base, to the canonical base.

\begin{theorem} \label{thm:E-DG}
Let $(\U, \cl)$ be a closure space with canonical base $(\U, \is_{DG})$ and let $x \in \U$.
For all $C \in \{\cl(B) \st B \in \gen_E(x)\}$ there exists an implication $P \imp C \setminus P$ in $\is_{DG}$ and a $E$-generator $A$ of $x$ such that $A \subseteq P$, $x \notin P$, and $\cl(A) = \cl(P) = C$.
\end{theorem}

\ifx\arxiv\undefined
\else
\begin{proof}
Let $C \in \{\cl(B) \st B \in \gen_E(x)\}$.
By Lemma \ref{lem:E-QC}, $Q = C \setminus \{y \st x \in \cl^b(y)\}$ is quasi-closed and spans $C$.
Let $P$ be a pseudo-closed set spanning $C$ such that $P \subseteq Q$.
As $P$ spans $C$ and $P$ is quasi-closed, it is $\cl^b$-closed.
Therefore, it contains a $\cl^b$-minimal spanning set $A$ of $C$. 
By Corollary \ref{cor:E-min-span-set}, $A$ is a $E$-generator of $x$.
This concludes the proof.
\ifx\arxiv\undefined
\qed
\fi
\end{proof}
\fi

Using Theorem \ref{thm:canonical} and Theorem \ref{thm:E-DG}, we obtain the following straightforward corollary connecting the $E$-base with other IBs.

\begin{corollary} \label{cor:E-IB}
Let $(\U, \cl)$ be a closure space and let $(\U, \is)$ be an IB for $(\U, \cl)$.
For all $x \in \U$ and all $C \in \{\cl(B) \st B \in \gen_E(x)\}$, there exists an implication $A \imp Y$ in $\is$ such that $x \notin A$ and $\cl(A) = C$.
\end{corollary}

Finally, we look at the particular case where all non-ji essential sets of $(\U, \cl)$ are incomparable.
We first prove a lemma that is the analogue for the $E$-generators and minimal non-ji essential sets of the fact that the pseudo-closed sets spanning an inclusion-wise minimal essential set are precisely its minimal spanning sets (see, e.g., \cite{distel2011complexity}). 

\begin{lemma} \label{lem:incomparable-nonji}
Let $C$ be an inclusion-wise minimal non-ji essential set.
Then, $\cl^b$ is a one-to-one correspondence between the $\cl^b$-minimal spanning sets of $C$ and its pseudo-closed sets.
Moreover, for any $\cl^b$-minimal spanning set $A$ of $C$ and any $x \in C \setminus \cl^b(A)$, $A$ is a $E$-generator of $x$.
\end{lemma}

\begin{proof}
We first show that for any $\cl^b$-minimal spanning set $A$ of $C$, $\cl^b(A)$ is pseudo-closed.
We begin by showing that it is quasi-closed.
Let $Q \subseteq \cl^b(A)$ such that $\cl(Q) \subset \cl(A)$.
By assumption, any essential set in $(\idl \cl(Q), \subseteq)$ is join-irreducible.
Therefore, any implication of the canonical base whose premise and conclusion are included in $\cl(Q)$
is binary.
It follows that for every $Q' \subseteq \cl(Q)$, $\cl(Q') = \cl^b(Q')$.
We obtain $\cl(Q) = \cl^b(Q) \subseteq \cl^b(A)$.
This makes $\cl^b(A)$ quasi-closed.
Assume for contradiction it is not pseudo-closed and let $P \subset \cl^b(A)$ be a pseudo-closed set spanning $C$.
Then, $P$ is $\cl^b$-closed, and thus, there exists a $\cl^b$-minimal spanning set $B$ of $C$ included in $P$.
We obtain $B \subseteq P \subset \cl^b(A)$, which contradicts $A$ being a $\cl^b$-minimal spanning set of $C$.
Therefore, $\cl^b(A)$ is indeed pseudo-closed.
As a consequence, we also obtain that each pseudo-closed set $P$ spanning $C$ equals $\cl^b(A)$ for some $\cl^b$-minimal spanning set $A$ of $C$.
Such a $A$ must moreover be unique as $(\U, \cl)$ is standard, so that for any two distinct $\cl^b$-minimal spanning sets $A, B$ of $C$, $\cl^b(A) \neq \cl^b(B)$.
We thus obtain that $\cl^b$ is a one-to-one correspondence between the $\cl^b$-minimal spanning sets of $C$ and its pseudo-closed sets.

It remains to argue that for any $\cl^b$-minimal spanning set $A$ of $C$, and any $x \in C \setminus \cl^b(A)$, $A$ is a $E$-generator of $x$.
By assumption, for every $C' \in \cs$ such that $C' \subset C$ and $x \in C$, all essential sets in $(\idl C', \subseteq)$ are ji.
Therefore, $x$ is prime in $(\idl C', \subseteq)$ as otherwise there would exist the closure of some $E$-generator of $x$ included in $C'$.
Since $\cl^b(A)$ is $\cl^b$-closed, $C$ is non-ji, and $x \in C \setminus \cl^b(A)$, $\cl^b(A)$ contains a $D$-generator of $x$, so that $x$ is not prime in $(\idl C, \subseteq)$.
Hence, $x$ is almost prime in $(\idl C, \subseteq)$.
As $A$ is a $\cl^b$-minimal spanning set $A$ of $C$ being a non-trivial $\cl^b$-generator of $x$, we deduce by Lemma~\ref{lem:E-generator} that $A$ is a $E$-generator of $x$.
This concludes the proof.
\end{proof}
In particular, if $(\U, \cl)$ is an atomistic closure space, we have $\cl^b(A) = A$ for any $A \subseteq \U$.
Thus, Lemma~\ref{lem:incomparable-nonji} directly yields the next corollary that draws another connection between the $E$-base and the canonical base in case all non-ji essential sets are incomparable.
\begin{corollary} \label{cor:incomparable-valid}
The $E$-base of a closure space $(\U, \cl)$ where the non-ji essential sets are pairwise incomparable is valid.
If moreover $(\U, \cl)$ is atomistic, then its canonical and $E$-base are equal.
\end{corollary}

\section{Semidistributive lattices} \label{sec:semidistributive}

This section is devoted to Theorem \ref{thm:E-valid-SD}.
We first give a key lemma.

\begin{lemma} \label{lem:pseudo-E}
Let $(\U, \cl)$ be a closure space with semidistributive lattice $(\cs, \subseteq)$.
Let $E$ be a non-ji essential set of $(\U, \cl)$ with associated pseudo-closed set $P$ and canonical spanning set $A$.
For each $C \prec E$ in $(\cs, \subseteq)$ and each $x$ in the canonical spanning set of $C$, $x \notin P$ implies that $A \imp x$ belongs to the $E$-base of $(\U, \cl)$.
\end{lemma}

\ifx\arxiv\undefined
\else
\begin{proof}
As a preliminary step, note that since $(\cs, \subseteq)$ is semidistributive, so is $(\idl C, \subseteq)$.
Moreover, $X \subseteq C$ is essential (resp.~quasi-closed, pseudo-closed) in $(\cs, \subseteq)$ if and only if it is in $(\idl C, \subseteq)$.
Therefore, in order to prove our lemma, it is sufficient to prove it for $E = \U$.

Suppose that $\U$ is essential and let $P$ be the pseudo-closed set spanning $\U$.
Let $A = \{a_1, \dots, a_k\}$ be the canonical spanning set of $\U$, and let $M_1, \dots, M_k$ be the corresponding coatoms given by Theorem \ref{thm:coatoms-SDj}.
Remark that since $P$ spans $\U$, and as $\U \neq \cl(z)$ for $z \in \U$ by assumption, $P$ is $\cl^b$-closed.
In particular, $A \subseteq P$.
For each $1 \leq i \leq k$, we define
\[ F_i = P \cap M_i. \]
We begin with a series of claims, the first of which is a useful property of the $F_i$'s.

\begin{claim} \label{claim:F-incomp}
For each $i$, $A \setminus \{a_i\} \subseteq F_i$ and $a_i \notin F_i$.
Thus, $F_j \nsubseteq F_i$ for each $i \neq j$.
\end{claim}

\begin{proof}
First, $A$ is a minimal spanning set of $\U$, so that $a_i \notin \cl(A \setminus \{a_i\})$.
Hence, $A \setminus \{a_i\} \subseteq \cl(A \setminus \{a_i\}) \subseteq M_i$ since $M_i$ is closed and it is the unique mi such that $a_i \upp M_i$ due to the primality of $a_i$.
Thanks to $A \subseteq P$, we readily have $A \setminus \{a_i\} \subseteq P$.
Thus, $A \setminus \{a_i\} \subseteq M_i \cap P = F_i$.
The fact that $a_i \notin F_i$ follows from $a_i \notin M_i$.
\renewcommand{\qed}{$\blacksquare$}
\ifx\arxiv\undefined
\qed
\fi
\end{proof}

\begin{claim} \label{claim:F-cap-M}
For each $i \neq j$, $F_i \cap (M_i \cap M_j) = F_j \cap (M_i \cap M_j)$.
\end{claim}

\begin{proof}
This follows from $F_i \cap (M_i \cap M_j) = P \cap M_i \cap (M_i \cap M_j) = P \cap M_j \cap (M_i \cap M_j) = F_j \cap (M_i \cap M_j)$.
\renewcommand{\qed}{$\blacksquare$}
\ifx\arxiv\undefined
\qed
\fi
\end{proof}

\begin{claim} \label{claim:MiMj-below-P}
For each $i \neq j$, $F_i \cap F_j = M_i \cap M_j \subseteq P$.
\end{claim}

\begin{proof}
By definition of $F_i, F_j$, we have $F_i \cap F_j \subseteq M_i \cap M_j$.
We show that the other inclusion also holds.
Assume for contradiction $F_i \cap F_j \subset M_i \cap M_j$.
Since $P$ is quasi-closed, $F_i$ and $F_j$ are closed.
Hence, $F_i \cap F_j$ is closed.
Let $x$ be an element of $M_i \cap M_j$ such that $F_i \cap F_j \downp x$.
Such an element must exist by assumption.
By Claim \ref{claim:F-cap-M}, we deduce $F_i \downp x$, $F_j \downp x$.
But then, there exists inclusion-wise maximal $C_i, C_j \in \cs$ such that $F_i \subseteq C_i \subseteq M_i$, $F_j \subseteq C_j \subseteq M_j$, and $C_i, C_j \downp x$.
Hence, we have $x \dpp C_i, C_j$.
However, by Claim \ref{claim:F-incomp} we have $a_i \in C_j \setminus C_i$ and $a_j \in C_i \setminus C_j$ as $F_i \subseteq C_i \subseteq M_i$ and $F_j \subseteq C_j \subseteq M_j$.
Thus, $C_i$ and $C_j$ are incomparable, and $C_i, C_j \dpp x$ contradicts $(\cs, \subseteq)$ being SDm due to Theorem \ref{thm:SD-arrows}.
Hence, such a $x$ cannot exist, and $F_i \cap F_j = M_i \cap M_j$.
Since $F_i \cap F_j \subseteq P$, we deduce that $M_i \cap M_j \subseteq P$ holds true.
\renewcommand{\qed}{$\blacksquare$}
\ifx\arxiv\undefined
\qed
\fi
\end{proof}

We go back to the proof of the lemma.
Recall that $A$ is the canonical spanning set of $\U$.
Since $P$ is pseudo-closed and spans $\U$, $P \subset \U$ must be true.
Moreover, $\U$ is not ji and $(\U, \cl)$ is standard by assumption so that $\U = \bigcup_{i = 1}^k M_i$.
Therefore, there exists $i$ such that $F_i \subset M_i$, that is $M_i \nsubseteq P$.
Since $P$ is quasi-closed, $M_i \nsubseteq P$ implies that $A_i \nsubseteq P$ where $A_i$ is the canonical spanning set of $M_i$.
Let $x \in A_i \setminus P$.
We show that $A$ and $x$ satisfy the conditions of Lemma \ref{lem:E-generator}.
Condition (1) is readily fulfilled by $A$.
Condition (2) follows from the fact that $x \in \cl(A) = \U$, $x \notin A$ as $A \subseteq P$ and $x \notin P$, and $\cl^b(A) \subseteq P$ since $P$ is quasi-closed.
It remains to show condition (3') holds.
First, by Claim \ref{claim:MiMj-below-P}, $x \notin M_j$ for $j \neq i$.
Thus, we need only show $x$ is prime in $(\idl M_i, \subseteq)$.
But since $x$ belongs to the canonical spanning set of $M_i$, it follows from Theorem \ref{thm:coatoms-SDj} that $x$ is prime in $(\idl M_i, \subseteq)$.
Henceforth, condition (3') holds. 
Since $A$, $x$ fulfill all the conditions of Lemma \ref{lem:E-generator}, we deduce that $A \imp x$ belongs to the $E$-base of $(\U, \cl)$, concluding the proof.
\ifx\arxiv\undefined
\qed
\fi
\end{proof}
\fi

We proceed to the proof of Theorem \ref{thm:E-valid-SD} that we first restate.

\ifx\arxiv\undefined
\begingroup
\def\thetheorem{\ref{thm:E-valid-SD}}
\begin{theorem}
The (aggregated) $E$-base of a standard closure space with semidistributive lattice is valid and minimum.
\end{theorem}
\addtocounter{theorem}{-1}
\endgroup
\else
\THMEvalidSD*
\fi

\begin{proof}
Let $(\U, \cs)$ be a standard closure space with SD lattice $(\cs, \subseteq)$.
Let $(\U, \is_E)$ be its $E$-base.
We need to show that for every $X \subseteq \U$, $\cl(X) = \is_E(X)$ where $\is_E(\cdot)$ is the closure operator induced by $(\U, \is_E)$.
By definition of the $E$-base, $\is_E(X) \subseteq \cl(X)$ for every $X \subseteq \U$.

We prove the Theorem using induction on closed sets ordered by inclusion. 
As the space is standard, it readily holds for $\emptyset$ and all atoms of $(\cs, \subseteq)$.
Let $C \in \cs$ and assume that for every $X \subseteq C$ such that $\cl(X) \subset C$, we have $\cl(X) = \is_E(X)$.
We argue that for every $Y \subseteq C$ such that $\cl(Y) = C$, we have $\is_E(Y) = C$.
If $Y = C$ then $\is_E(Y) = C$ readily holds since $\is_E(Y) \subseteq \cl(Y)$.
Thus, we can assume that $Y \subset C$.
We distinguish three cases:
\begin{enumerate}[(1)]
\item $C$ is not essential (hence not join-irreducible).
Since $C$ is not essential, $Y$ is not quasi-closed.
Therefore, there exists $X \subseteq Y$ such that $\cl(X) \subset \cl(Y)$ and $\cl(X) \nsubseteq Y$.
However, by inductive hypothesis, $\cl(X) = \is_E(X)$ and hence $\is_E(X) \nsubseteq Y$ so that $Y \subset \is_E(Y)$. 
Since this holds for every proper subset $Y$ of $C$ spanning $C$ and as $\is_E(Y) \subseteq \cl(Y)$ for all such $Y$'s, we deduce that $\is_E(Y) = \cl(Y) = C$ as required.

\item $C$ is join-irreducible.
Since the system is standard, we thus have that $C = \cl(x)$ for a unique $x \in \U$.
Moreover, $Y$ contains $x$ since $\{x\}$ is the unique minimal spanning set of $C$.
By definition $\is_E$ contains an implication $x \imp y$ for each $y$ such that $y \in \cl(x)$.
we deduce that $\cl(Y) = \is_E(Y) = C$ holds true.

\item $C$ is essential but not join-irreducible.
As all non-trivial binary implications of $(\U, \cl)$ belongs to $\is_E$, we have that $\is_E(Y)$ contains $A$, the canonical spanning set of $C$.
To show that $\is_E(Y) = C$, we prove that for each of the predecessors $C_i$, \dots, $C_k$ of $C$, the canonical spanning set $A_i$ of $C_i$ is included in $\is_E(Y)$.
On this purpose, we first argue that $\is_E(Y)$ is quasi-closed with respect to $\cl$.
Let $X \subseteq \is_E(Y)$ such that $\cl(X) \subset \cl(\is_E(Y)) = \cl(Y) = C$.
By inductive hypothesis, $\is_E(X) = \cl(X)$ so that $\cl(X) = \is_E(X) \subseteq \is_E(\is_E(Y)) = \is_E(Y)$.
Therefore, $\is_E(Y)$ is quasi-closed and in particular $P \subseteq \is_E(Y)$ with $P$ the unique pseudo-closed set spanning $C$ obtained from Proposition \ref{prop:SDj-UC}.
Now, let $C_i$ be a predecessor of $C$ and consider its canonical spanning set $A_i$.
Let $x \in A_i$.
If $x \in P$, then $x \in \is_E(Y)$ by previous argument.
If $x \notin P$, then by Lemma \ref{lem:pseudo-E} we have $A \imp x \in \is_E$.
Because $A \subseteq P \subseteq \is_E(Y)$, we deduce that $x \in \is_E(Y)$.
Therefore $A_i \subseteq \is_E(Y)$ for each $1 \leq i \leq k$.

Now, $\cl(A_i) = C_i \subset C$ for each $1 \leq i \leq k$.
By inductive hypothesis, $\is_E(X) = \cl(X)$ for each $X \subseteq C$ such that $\cl(X) \subset C$.
Henceforth, $C_i = \is_E(A_i) \subseteq \is_E(Y)$.
As $C$ is not join-irreducible and $(\U, \cl)$ is standard, we have $C = \bigcup_{i = 1}^{k} C_i \subseteq \is_E(Y) \subseteq \cl(Y) = C$.
We deduce that $\is_E(Y) = C$ for every subset $Y$ of $C$ such that $\cl(Y) = C$.
\end{enumerate}
In all three cases, we have that $\is_E(Y) = \cl(Y)$ for every $Y \subseteq C$ such that $\cl(Y) = C$.
We deduce using induction that the $E$-base is valid.

It remains to show that in its aggregated form $(\U, \is_E)$ is minimum.
On this purpose, we identify a bijection between $\is_E$ and $\is_{DG}$.
Observe first that they have the same binary part by definition.
Hence, we need only analyze their non-binary parts.
By Lemma \ref{lem:E-QC}, the closure of a $E$-generator is essential.
By Lemma \ref{lem:pseudo-E}, the canonical spanning set of each non-ji essential set is a $E$-generator.
Since $(\cs, \subseteq)$ is semidistributive, we deduce by Theorem \ref{thm:SDj-canonical} that there is a bijection between $E$-generators and non-ji essential sets.
Again by semidistributivity of $(\cs, \subseteq)$, there is a bijection between essential sets and pseudo closed sets by virtue of Proposition \ref{prop:SDj-UC}.
We get a bijection between $E$-generators and non-singleton pseudo-closed sets.
By Theorem~\ref{thm:canonical}, $(\U, \is_E)$ is thus minimum once aggregated.
\end{proof}

We remark that both meet- and join-semidistributivity are important.
Indeed, in general, neither meet-semidistributive nor join-semidistributive have valid $E$-base.
The lattice of Example \ref{ex:leaf} is join-semidistributive but not meet-semidistributive and the $E$-base of the corresponding closure space is not valid.
Below, we give an example of a closure space whose lattice is meet-semidistributive and its $E$-base not valid.

\begin{example} \label{ex:SDm-fail}
Let $(\U, \cl)$ be the closure space of the meet-semidistributive lattice of Figure \ref{fig:SDm-fail}.
We have:

\begin{align*}
\is_{DG} & = \{\agset{e \imp ac}, \agset{f \imp ad}, \agset{c \imp a}, \agset{d \imp a}\} & (\text{binary}) \\ 
& \cup \{\agset{acde \imp f}, \agset{acdf \imp e} \} & (\agset{acdef}) \\
& \cup \{\agset{ab \imp cdef} \} & (\agset{abcdef}) \\
& & \\
\is_E & = \{\agset{e \imp ac}, \agset{f \imp ad}, \agset{c \imp a}, \agset{d \imp a}\} & (\text{binary}) \\ 
& \cup \{\agset{de \imp f}, \agset{cf \imp e}\} & (\agset{acdef}) \\ 
& \cup \{\agset{ab \imp cd}\} & (\agset{abcdef})
\end{align*}

The $E$-base $(\U, \is_E)$ is not valid since $\agset{ab}$ does not generate $\agset{ef}$.
This example illustrates the importance of join-semidistributivity in Lemma \ref{lem:pseudo-E}.
Indeed, $\agset{acdef}$ is a predecessor of the essential set $\agset{abcdef}$ which has no canonical representation.
Its minimal spanning sets are $\agset{de}$, $\agset{cf}$ and $\agset{ef}$ but none of them comprises only prime elements of $(\idl \agset{acdef}, \subseteq)$, which are $a$, $c$ and $d$.
Thus, the $E$-generator and pseudo-closed set $\agset{ab}$ does not reach any spanning set of $\agset{acdef}$ in $\is_E$.

\begin{figure}[ht!]
    \centering
    \includegraphics[scale=\FIGsdm]{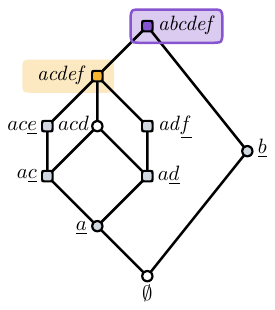}
    \includegraphics[width=\textwidth, page=2]{figures/captions.pdf}
    \caption{The meet-semidistributive lattice of Example \ref{ex:SDm-fail}. 
    The essential set $\agset{abcdef}$ is faulty, while $\agset{acdef}$ is not. The $E$-base is thus not valid.}
    \label{fig:SDm-fail}
\end{figure}
\end{example}

Note that in Example~\ref{ex:SDm-fail}, even though the $E$-base is not valid, all essential sets are spanned by some $E$-generator.
One could thus expect that in meet-semidistributive lattices, all essential sets are captured by the $E$-base.
We conclude the section with a counter-example to this claim which turns out to be, more than meet-semidistributive, upper-bounded and join-distributive.

\begin{example} \label{ex:fail-jdis}
We consider the lattice $(\cs, \subseteq)$ of Figure \ref{fig:fail-jdis}.
This lattice is not only join-distributive but also upper-bounded as the dual $D^*$ of the $D$-relation on meet-irreducible closed sets is without cycles.
This can be observed in Figure \ref{fig:fail-jdis} where the $D^*$ relation is represented as a directed graph which has no (directed) cycles.
As for the implications, the canonical base and $E$-base read as follows:

\begin{align*}
\is_{DG} & = \{\agset{c \imp ab}, \agset{b \imp a}, \agset{g \imp def}, \agset{f \imp de},
    \agset{d \imp e}, \agset{h \imp ad}\} & (\text{binary}) \\ 
& \cup \{\agset{abd \imp e}, \agset{ade \imp b}\} & (\agset{abde}) \\ 
& \cup \{\agset{abcde \imp h}, \agset{abdeh \imp c}\} & (\agset{abcdeh}) \\ 
& \cup \{\agset{abdefg \imp ch}, \agset{abcdefh \imp g}\} & (\agset{abcdefgh}) \\ 
& & \\
\is_E & = \{\agset{c \imp ab}, \agset{b \imp a}, \agset{g \imp def}, \agset{f \imp de},
    \agset{d \imp e}, \agset{h \imp ad}\} & (\text{binary}) \\ 
& \cup \{\agset{bd \imp e}, \agset{ae \imp b} \} & (\agset{abde}) \\
& \cup \{\agset{bh \imp c}, \agset{he \imp c}, \agset{cd \imp h}\} & (\agset{abcdeh}) \\ 
& \cup \{\agset{cf \imp g}, \agset{hf \imp g}\} & (\agset{abcdefgh})
\end{align*}

The pseudo-closed set $\agset{abdefg}$ is faulty since $c$, $h$ are not almost prime in $(\cs, \subseteq)$.
\begin{figure}[ht!]
    \centering
    \includegraphics[scale=\FIGjdis]{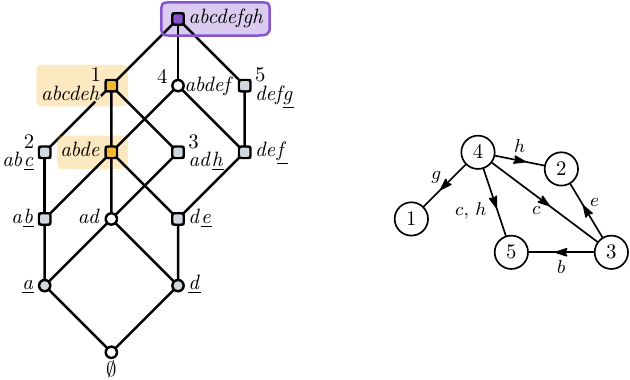}
    \includegraphics[width=\textwidth, page=2]{figures/captions.pdf}
    \caption{An upper-bounded join-distributive closure lattice where $E$-base is not valid (Example~\ref{ex:fail-jdis}).
    Meet-irreducible closed sets are labeled by numbers.
    The labeled directed graph of the $D^*$ relation is pictured on the right. 
    An arc between two mi closed sets indicates that they are related via $D^*$ and the label indicates the elements of the ground set that witness this relation.
    For instance, we have $4 D^* 3$ as $4 \upp c \downp 3$ which translates into an arc from $4$ to $3$ labeled by $c$.
    This directed graph is acyclic, which proves that the lattice is upper-bounded.
    The essential set $\agset{abcdefgh}$ is faulty because the pseudo-closed set $\agset{abdefg}$ is: it does not subsume any $E$-generator spanning $\agset{abcdefgh}$.}
    \label{fig:fail-jdis}
\end{figure}
\end{example}

\section{Modular lattices} \label{sec:modular}

In this section, we discuss the $E$-base of closure spaces with modular lattices.

We fix a standard closure space $(\U, \cl)$ with modular closure lattice $(\cs, \subseteq)$.
Recall that if $C$ is a closed set, $C_*$ is the intersection of the predecessors of $C$ (if any), and that $[C_*, C]$ is the interval of $(\cs, \subseteq)$ with minimal closed set $C_*$ and maximal closed set $C$.
We start with an example showing that in general, modular lattices do not have valid $E$-base.

\begin{example} \label{ex:modular-prototypical-fail}
Consider the closure space $(\U, \cl)$ associated to the closure lattice of Figure~\ref{fig:modular-prototypical-fail}.
We have:

\begin{align*}
\is_{DG} & = \{d \imp c, e \imp c\} & (\text{binary}) \\ 
& \cup \{ab \imp c, ac \imp b, bc \imp a\} & (abc) \\
& \cup \{abcd \imp e, abce \imp d, cde \imp ab\} & (abcde) \\
& & \\
\is_{E} & = \{d \imp c, e \imp c\} & (\text{binary}) \\ 
& \cup \{ab \imp c, ac \imp b, bc \imp a\} & (abc) \\
& \cup \{ad \imp e, bd \imp e, ae \imp d, be \imp d\} & (abcde)
\end{align*}

The implication $cde \imp ab$ does not follow from the $E$-base as $de$ is not a $E$-generator of $a$ or $b$.
This makes the essential set $abcde$ as well as the pseudo-closed set $cde$ faulty and the $E$-base non-valid. 
\begin{figure}[ht!]
    \centering
    \includegraphics[scale=\FIGmodproto]{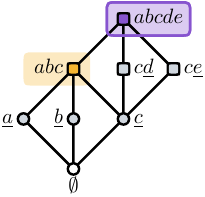}
    \includegraphics[width=\textwidth, page=2]{figures/captions.pdf}
    \caption{A modular lattice where the $E$-base is non-valid (Example~\ref{ex:modular-prototypical-fail}). While the non-ji essential set $abc$ is not faulty, $abcde$ is because $de$ is not a $E$-generator of $a$ or $b$.}
    \label{fig:modular-prototypical-fail}
\end{figure}
\end{example}
Example~\ref{ex:modular-prototypical-fail} is in fact a minimal representative of a configuration that characterizes the validity of the $E$-base in modular lattices: there exists a (non-ji) essential set $C$ and a predecessor $C'$ of $C$ such that $\card{C \setminus C'} > 1$.
To prove that the $E$-base of $(\U, \cl)$ is valid if and only if this situation does not occur, we will analyse how do quasi-closed, pseudo-closed and essential sets relate to $E$-generators and almost prime elements in $(\U, \cl)$.

As as a preliminary step, let us remind from Proposition~\ref{prop:prop-psi} that if $C$ is a non-ji essential set with predecessors $C_1, \dots, C_m$, then $[C_*, C]$ is a diamond and the quasi-closed spanning sets of $C$ are of the form:
\[ \bigcup_{i \in I} C_i \quad \text{with } I \subseteq \{1, \dots, m\} \text{ and } 1 < \card{I} < m. \] 
Another useful observation is that, for any non-empty $I \subseteq \{1, \dots, m\}$ we have:
\begin{equation*} 
\bigcup_{i \in I} C_i = C_* \cup \bigcup_{i \in I}^m C_i \setminus C_* 
\end{equation*} 
where the sets $C_i \setminus C_*$ are non-empty and pairwise disjoint.
We now look at the almost prime elements in $(\U, \cl)$.

\begin{lemma} \label{lem:almot-prime-modular}
Let $C \in \cs$ be non-ji essential set and let $C_1, \dots, C_m$ be the predecessors of $C$.
An element $x$ is almost prime in $(\idl C, \subseteq)$ if and only if there exists a unique $C_i$ such that $C_i \setminus C_* = \{x\}$. 
\end{lemma}

\begin{proof}
We start with the if part.
Assume there exists a unique $C_i$ such that $C_i \setminus C_* = \{x\}$.
Since $C$ is essential and $(\cs, \subseteq)$ is modular, we have by Proposition~\ref{prop:wild-modular} that $x \notin C_j$ for every $C_j \neq C_i$.
Given that $x \in C$ and $C$ has more than $3$ predecessors, we deduce that $x$ is not prime in $(\idl C, \subseteq)$.
Moreover, $C_* = C_i \setminus \{x\}$ entails that $x$ is prime in $(\idl C_i, \subseteq)$.
We deduce that $x$ is indeed almost prime in $(\idl C, \subseteq)$.

We move to the only if part.
Assume that $x$ is almost prime in $(\idl C, \subseteq)$.
Because $(\U, \cl)$ is standard and $C$ is non-ji, there exists at least one predecessor of $C$ that contains $x$.
Without loss of generality, assume that $x \in C_1$.
By Lemma~\ref{lem:E-QC}, $Q = C \setminus \{y \st x \in \cl^b(y)\}$ is quasi-closed.
From Proposition~\ref{prop:wild-modular}, $C_* \subseteq Q$.
Hence, $y \in C_1 \setminus C_*$ for every $y \in C$ such that $x \in \cl^b(y)$.
In particular, we thus have $Q = \bigcup_{i > 1} C_i$.
As $C_1 = C_* \jn \cl(y)$ and $C_* \prec C$, modularity entails $\cl(y) \cap C_* \prec \cl(y)$.
Since $(\U, \cl)$ is standard, $\cl(y)$ must be ji.
As $x \in \cl(y)$ by definition of $Q$, $x = y$ must hold.
We conclude that $Q = C \setminus \{x\}$.
Applying the fact that the $C_i \setminus C_*$'s are pairwise disjoint, we finally obtain:
\[ 
C \setminus Q = \bigcup_i C_i \setminus \bigcup_{j > 1} C_j = C_1 \setminus C_* = \{x\}
\] 
\end{proof}

\begin{lemma} \label{lem:mod-E-closure}
For every quasi-closed set $Q$ spanning $C$, we have:
\[ 
\is_E(Q) = Q \cup \{x \st \card{C_i \setminus C_*} = \{x\}\}
\] 
\end{lemma}

\begin{proof}
We show double-inclusion and start with the $\supseteq$ part.
We readily have $Q \subseteq \is_E(Q)$.
Now let $x \in C$ be such that $\{x\} = C_i \setminus C_*$ for some predecessor $C_i$ of $C$.
By Lemma~\ref{lem:almot-prime-modular}, $x$ is almost prime in $(\idl C, \subseteq)$, which makes any $\cl^b$-minimal spanning set of $C$ whose closure via $\cl^b$ does not contain $x$ a $E$-generator of $x$ by Corollary~\ref{cor:E-min-span-set}.
As $Q$ is $\cl^b$-closed and $x \notin Q$, it includes such a $\cl^b$-minimal spanning set of $C$ and $x \in \is_E(Q)$ follows.

We move to the $\subseteq$ part, which we prove using contradiction.
Suppose that $Q' = Q \cup \{x \st \card{C_i \setminus C_*} = \{x\}\}$ is not $\is_E$-closed.
Then, there exists $y \in C \setminus Q'$ and a $E$-generator $A$ of $y$ such that $A \subseteq Q'$ and $y \notin Q'$.
By Proposition~\ref{prop:wild-modular}, there exists a (unique) predecessor $C_j$ of $C$ such that $y \in C_j \setminus C_*$.
According to the previous paragraph, it must be that $\cl(A) \subseteq C_j$ as otherwise $\cl(A) = C$ which would imply $y \in \{x \st \card{C_i \setminus C_*} = \{x\}\}$.
Moreover, $y \notin Q'$ implies $y \notin C_*$ and hence $A \cap (C_j \setminus C_*) \neq \emptyset$.
However, $C_j \nsubseteq Q$ entails $Q \cap (C_j \setminus C_*) = \emptyset$ again by Proposition~\ref{prop:wild-modular}.
We deduce that there exists $z \in A \cap (C_j \setminus C_*)$, $z \neq y$, such that $z \in Q' \setminus Q$.
But $z \in Q' \cap (C_j \setminus C_*)$ entails $C_j \setminus C_* = \{z\}$ by definition of $Q'$, a contradiction with $y \neq z$ and $y \in C_j \setminus C_*$.
We deduce that $Q'$ is indeed $\is_E$-closed, which concludes the proof.
\end{proof}

We thus obtain the subsequent characterization of faulty pseudo-closed sets.

\begin{lemma} \label{lem:modular-faulty}
Let $C$ be a non-ji essential set with predecessors $C_1, \dots, C_m$.
Then a pseudo-closed set $P = C_i \cup C_j$ spanning $C$ is faulty if and only if there exists some $C_k \neq C_i, C_j$ such that $\card{C_k \setminus C_*} > 1$.
\end{lemma}

\begin{proof}
The pseudo-closed set $P$ is faulty iff there exists $x \in C \setminus P$ such that $x \notin \is_E(P)$.
By Lemma~\ref{lem:mod-E-closure}, $\is_E(P) = P \cup \{x : \card{C_i \setminus C_*} = \{x\}\}$.
Thus $x \in C \setminus P$ and $x \notin \is_E(P)$ is in turn equivalent to the fact that there exists a (unique) predecessor $C_k \neq C_i, C_j$ of $C$ such that $x \in C_k \setminus C_*$ and $C_k \setminus C_* \neq \{x\}$, i.e., $\card{C_k \setminus C_*} > 1$ as $C_k \setminus C_*$ is non-empty. 
\end{proof}

We finally derive Theorem~\ref{thm:modular-Ebase} as a direct corollary of Lemma~\ref{lem:modular-faulty}.

\ifx\arxiv\undefined
\begingroup
\def\thetheorem{\ref{thm:modular-Ebase}}
\begin{theorem}
The (aggregated) $E$-base of a standard closure space with modular lattice is valid if and only if for every essential set $C$ and any predecessor $C'$ of $C$, $\card{C' \setminus C_*} = 1$, where $C_*$ is the intersection of the predecessors of $C$.
\end{theorem}
\addtocounter{theorem}{-1}
\endgroup
\else
\THMmodular*
\fi

We highlight two consequences of Lemma~\ref{lem:almot-prime-modular} and Theorem~\ref{thm:modular-Ebase}.
First, one can observe that in Example~\ref{ex:modular-prototypical-fail}, even though the $E$-base is not valid, all the essential sets are spanned by some $E$-generator.
However, by taking a diamond and gluing to each of its join-irreducible elements another diamond, thus repeating the characteristic situation of Theorem~\ref{thm:modular-Ebase}, we will obtain an essential set not spanned by any $E$-generator.
This is Example~\ref{ex:fail-mod} below.

\begin{example} \label{ex:fail-mod}
Consider the closure space $(\U, \cl)$ associated to the lattice of Figure \ref{fig:fail-mod}.
The lattice is modular.
The canonical and $E$-base are given by:

\begin{align*}
\is_{DG} & = \{ \agset{d \imp ab}, \agset{e \imp ab}, \agset{f \imp ac}, \agset{g \imp ac},
    \agset{h \imp bc}, \agset{i \imp bc} \} & (\text{binary}) \\ 
& \cup \{\agset{abcd \imp e}, \agset{abce \imp d}, \agset{abde \imp c}\} & (\agset{abcde}) \\ 
& \cup \{\agset{abcf \imp g}, \agset{abcg \imp f}, \agset{acfg \imp b}\} & (\agset{abcfg}) \\
& \cup \{\agset{abch \imp i}, \agset{abci \imp h}, \agset{bchi \imp a}\} & (\agset{abchi}) \\
& \cup \{\agset{abcdefg \imp hi}, \agset{abcdehi \imp fg}, \agset{abcfghi \imp de}\} & (\agset{abcdefghi}) \\
& & \\
\is_{E} & = \{ \agset{d \imp ab}, \agset{e \imp ab}, \agset{f \imp ac}, \agset{g \imp ac},
    \agset{h \imp bc}, \agset{i \imp bc} \} & (\text{binary}) \\ 
& \cup \{\agset{cd \imp e}, \agset{ce \imp d}, \agset{de \imp c}\} & (\agset{abcde}) \\ 
& \cup \{\agset{bf \imp g}, \agset{bg \imp f}, \agset{fg \imp b}\} & (\agset{abcfg}) \\
& \cup \{\agset{ah \imp i}, \agset{ai \imp h}, \agset{hi \imp a}\} & (\agset{abchi}) \\
\end{align*}

While the essential sets $\agset{abcde}$, $\agset{abcde}$, $\agset{abcfg}$ and $\agset{abchi}$ are not faulty, there is no $E$-generator to span the essential set $\agset{abcdefghi}$, since none of the elements $d$, $e$, $f$, $g$, $h$ and $i$ are almost prime in $(\cs, \subseteq)$.
Connecting to Theorem~\ref{thm:modular-Ebase}, there is 4 elements of difference between $\agset{abcdefghi}$ and each of its predecessors.
\begin{figure}[ht!]
    \centering
    \includegraphics[scale=\FIGmod]{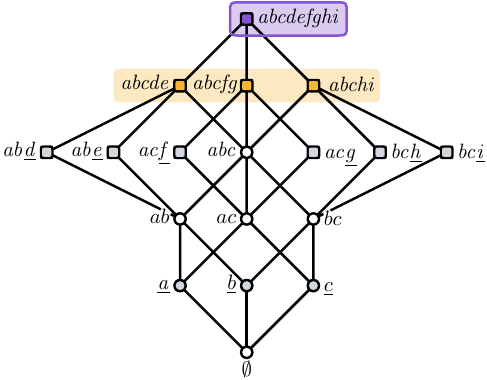}
    \includegraphics[width=\textwidth, page=2]{figures/captions.pdf}
    \caption{A modular closure lattice where the essential set $abcdefghi$ is not spanned by  any $E$-generator. 
    This makes the $E$-base non valid.}
    \label{fig:fail-mod}
\end{figure}
\end{example}

The second consequence deals with the case where $(\cs, \subseteq)$ is atomistic and modular, or equivalently geometric and modular.
In this case, $(\cs, \subseteq)$ is also coatomistic, and $C_* = \emptyset$ for every $C \in \cs$.
By Proposition~\ref{prop:wild-modular}, each essential set must thus be of height $2$ in $(\cs, \subseteq)$.
This makes all essential sets incomparable. 
Using Theorem~\ref{thm:modular-Ebase} (or Corollary~\ref{cor:incomparable-valid}), we deduce:
\begin{corollary} \label{cor:atmod-valid}
The $E$-base of a closure space with atomistic modular lattice is valid and equals its canonical base.
\end{corollary}

In particular, given that the lattice of linear subspaces of a projective space is atomistic and modular---equivalently modular and geometric---\cite[Theorem 424]{gratzer2011lattice}, Corollary~\ref{cor:atmod-valid} specifies to:
\begin{corollary} \label{cor:proj}
The lattice of linear subspaces of a projective space has valid $E$-base.
\end{corollary}

\section{Geometric lattices} \label{sec:geometric}

In this section, we characterize geometric lattices with valid $E$-base.
First, we illustrate that geometric lattices might not have valid $E$-base in general.

\begin{example} \label{ex:fullfail-geom}
Consider the geometric closure lattice $(\cs, \subseteq)$ of Figure~\ref{fig:fullfail-geom}.
It is the direct product of two diamonds, $abc$ and $\mathit{def}$, but truncated at height $2$.
We have:

\begin{align*}
\is_{DG} & = \{\mathit{ab \imp c}, \mathit{ac \imp b}, \mathit{bc \imp a}\} & (\mathit{abc}) \\ 
& \cup \{\mathit{de \imp f}, \mathit{df \imp e}, \mathit{ef \imp e}\} & (\mathit{def}) \\
& \cup \{\mathit{abcd \imp ef}, \mathit{abce \imp df}, \mathit{abcf \imp de}, \mathit{adef \imp bc}, \mathit{bdef \imp ac}, \mathit{cdef \imp ab} \} & (\mathit{abcdef}) \\ 
& & \\ 
\is_E & = \{\mathit{ab \imp c}, \mathit{ac \imp b}, \mathit{bc \imp a}\} & (\mathit{abc}) \\ 
& \cup \{\mathit{de \imp f}, \mathit{df \imp e}, \mathit{ef \imp e}\} & (\mathit{def})
\end{align*}
Hence, $\mathit{abcdef}$ is essential, but not spanned by any $E$-generator as none of the elements are almost prime. 
This makes this essential set, as well as the pseudo-closed sets, faulty.
\begin{figure}[ht!]
    \centering
    \includegraphics[scale=\FIGgeomfull]{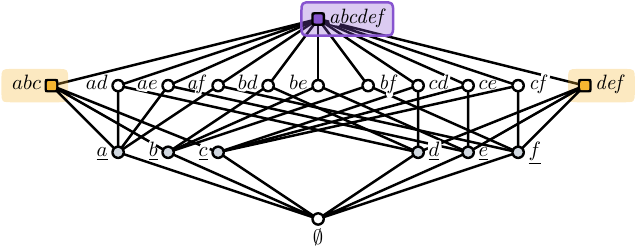}
    \includegraphics[width=\textwidth, page=2]{figures/captions.pdf}
    \caption{A geometric closure lattice where $E$-base is not valid as $\mathit{abcdef}$ is not spanned by any $E$-generator.}
    \label{fig:fullfail-geom}
\end{figure}
\end{example}
We will prove that the $E$-base of a geometric lattice is in fact valid if and only if all its essential sets are incomparable.
Note that the if part already follows from Corollary~\ref{cor:incomparable-valid}.
Now, let us fix some exchange closure space $(\U, \cl)$ that admits comparable essential sets.
It is sufficient to place ourselves in the case where $\U$ is essential and all other essential sets are incomparable.
If $\cc{M} = (\U, \cc{I})$ is the matroid associated to $(\U, \cl)$, recall that $\cc{B}$ denotes its bases and $\cc{K}$ its circuits.
Let $b$ be the size of a base.

We argue that the $E$-base of $(\U, \cl)$ is not valid because $\U$ is faulty.
We may assume that there are some almost prime elements in $(\cs, \subseteq)$ as otherwise the result is clear.
Let us observe though that such almost prime elements can indeed exist, as illustrated by Example~\ref{ex:fail-geom} below.

\begin{example} \label{ex:fail-geom}
Consider the closure space $(\U, \cl)$ corresponding to the lattice $(\cs, \subseteq)$ of Figure~\ref{fig:fail-geom}.
The lattice is geometric, and the canonical base and the $E$-base read as follows:
\begin{align*}
\is_{DG} & = \{\mathit{ab \imp c}, \mathit{ac \imp b}, \mathit{bc \imp a}\} & (\mathit{abc}) \\ 
& \cup \{\mathit{abcd \imp e}, \mathit{abce \imp d}, \mathit{ade \imp bc}, \mathit{bde \imp ac}, \mathit{cde \imp ab}\} & (\mathit{abcde}) \\ 
& & \\ 
\is_{E} & = \{\mathit{ab \imp c}, \mathit{ac \imp b}, \mathit{bc \imp a}\}\} & (\mathit{abc}) \\
& \cup \{\mathit{abd \imp e}, \mathit{abe \imp d}, \mathit{acd \imp e}, \mathit{ace \imp d}, \mathit{bcd \imp e}, \mathit{bce \imp d}\} & (\mathit{abcde})
\end{align*}
Notice that both existing essential sets $\agset{abcde}$ and $\agset{abc}$ are spanned by some $E$-generator.
However, the implications $\agset{ade \imp bc}$, $\agset{bde \imp ac}$ and $\agset{cde \imp ab}$ do not follow from the $E$-base as the corresponding pseudo-closed sets are not subsumed by any $E$-generator.
This makes $\agset{abcde}$ faulty and the $E$-base of $(\U, \cl)$ not valid.
\begin{figure}[ht!]
    \centering
    \includegraphics[scale=\FIGgeom]{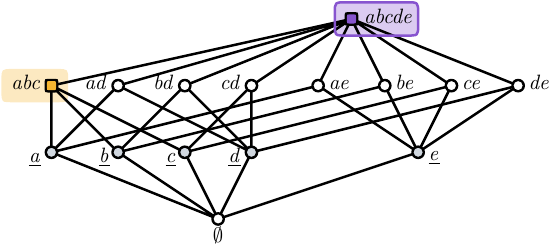}
    \includegraphics[width=\textwidth, page=2]{figures/captions.pdf}
    \caption{A geometric closure lattice where $E$-base is not valid.
    The pseudo-closed sets $\agset{ade}$, $\agset{bde}$ and $\agset{cde}$ are not captured by the $E$-base.}
    \label{fig:fail-geom}
\end{figure}
\end{example}

Hence, let $A = \{a_1, \dots, a_k\}$ be the set of almost prime elements of $(\cs, \subseteq)$.
The subsequent remark will be useful. 
It is a consequence of the fact that since $a$ is almost prime, any of its minimal generator must span $\U$.
\begin{remark} \label{rem:matroid-AP}
The minimal generators of any almost prime element $a$ are precisely the bases of $\cc{M}$ that do not contain $a$.
\end{remark}

We first prove that there exists a base $B^*$ that contains as much almost prime elements as possible.

\begin{lemma}
There exists a base $B^*$ such that $\card{B^* \cap A} = \min \{b, \card{A}\}$.
\end{lemma} 

\begin{proof}
Let $B$ be any base
If it satisfies the requirement, then $B^*=B$, otherwise, consider $m = \card{B \cap A} < b,|A|$.
We show that there exists a base $B'$ such that $\card{B' \cap A} = m+1$. 
From the definition of $B$, we deduce that $A \nsubseteq B$ and $B \nsubseteq A$.
Let $a \in A \setminus B$ and $x \in B \setminus A$.
We have $a \in \cl(B)$ as $B$ spans $\U$.
Moreover, $B$ is a base of $\cc{M}$ and $a$ is almost prime in $(\cs, \subseteq)$.
We deduce that $B$ is a $E$-generator of $a$ and hence a minimal generator of $a$.
In particular, there exists a circuit $K$ such that $B = K \setminus \{a\}$.
Consider $B' = (K \setminus \{x\}) \cup \{a\}$.
We have that $B'$ is a minimal generator of $x$ by the exchange property of $\cl$.
Therefore, $B \subseteq \cl(B')$ and hence $\cl(B') = \cl(B) = \U$.
Given that all minimal spanning sets of $\U$---all bases of $\cc{M}$---must have size $b$ and $\card{B} = \card{B'}$, we deduce that $B'$ is a base of $\cc{M}$.
Moreover $\card{B' \cap A} = m + 1$ by construction.
Applying this argument repeatedly, we will reach a base $B^*$ that satisfies the statement.
\end{proof}

It remains to prove that such a base $B^*$ can be used to exhibit a spanning set of $\U$ that is closed with respect to the $E$-base.

\begin{lemma} \label{lem:matroid-AP}
Let $Y = \max_\subseteq \{A, B^*\}$.
Then $\is_E(Y) \subset \cl(Y) = \U$.
\end{lemma}

\begin{proof}
We consider the two cases.
Assume first that $Y = A$.
Observe that $Y$ is not closed as $\U$ contains non-almost prime elements by assumption and $\cl(Y) = \U$.
We argue that $Y$ is $\is_E$-closed.
Let $K$ be a circuit and let $x \in K$ such that $K \setminus x \imp x \in \is_E$.
If $x \in A$, then $Y = A$ satisfies the implication $K \setminus x \imp x$.
Hence, assume $x \notin A$.
Then, $x$ is not almost prime in $(\cs, \subseteq)$.
Therefore, $K \setminus x \imp x \in \is_E$ entails $\cl(K) \subset \U$.
As for any $y \in K$, $K \setminus y$ is a minimal generator of $y$, we deduce from $\cl(K) \subset \U$ and Remark~\ref{rem:matroid-AP} that $K \cap A = \emptyset$.
Therefore, $A$ vacuously satisfies the implication $K \setminus x \imp x$.
We deduce that $Y = A$ is $\is_E$-closed as required.

We move to the second case $Y = B^*$.
We first show that $B^*$ is a minimal generator of some $y \notin B^*$.
Let $a \in A$.
Since $a$ is almost prime, it admits a minimal generator $B$ which is a base of $\cc{M}$ (different from $B^*$) by Remark~\ref{rem:matroid-AP}.
We have $a \in B^* \setminus B$.
Hence, applying the base exchange axiom, there exists $y \in B \setminus B^*$ such that $B' = (B^* \setminus \{a\}) \cup \{y\}$ is a base of $\cc{M}$.
In particular, $y \in B \setminus B^*$ implies $y \notin A$.
But then, $a \notin B'$ implies that $B'$ is of the form $B' = K \setminus \{a\}$ for some circuit $K$ containing $a$ and $y$ again by Remark~\ref{rem:matroid-AP}.
Therefore, $B^* = (B' \setminus \{y\}) \cup \{a\} = K \setminus \{y\}$ is a minimal generator of $y$.

We now prove that $y \notin \is_E(Y) =\is_E(B^*)$.
Recall that $\is_E \subseteq \is_{cd}$.
To prove our claim, we show that for any inclusion-wise minimal subset $\is'$ of $\is_{cd}$ that allows to reach $y$ from $B^*$ using forward chaining, $\is' \nsubseteq \is_E$.
Since $B^*$ is a minimal generator of $y$ and $A \subseteq B^*$, any such $\is'$ must contain, for each $a \in A$, an implication $A' \imp x$ where $a \in A'$ and $x \notin A$.
As $(\U, \cl)$ has the exchange property, $(A' \setminus \{a\}) \cup \{x\}$ is thus a minimal generator of $a$ since $A' \imp x \in \is_{cd}$ .
Because $x$ is not almost prime in $(\cs, \subseteq)$, we deduce $A' \imp x \notin \is_E$.
Therefore, $\is' \nsubseteq \is_E$.
This proves that $y \notin \is_E(Y)$, and concludes the proof.

\end{proof}

Combining Corollary~\ref{cor:incomparable-valid} and Lemma~\ref{lem:matroid-AP}, we obtain the desired characterization of geometric lattices with valid $E$-base.

\ifx\arxiv\undefined
\begingroup
\def\thetheorem{\ref{thm:matroid-valid}}
\begin{theorem}
The (aggregated) $E$-base of a standard closure space with modular lattice is valid if and only if for every essential set $C$ and any predecessor $C'$ of $C$, $\card{C \setminus C'} = 1$.
\end{theorem}
\addtocounter{theorem}{-1}
\endgroup
\else
\THMgeometric*
\fi

Theorem~\ref{thm:matroid-valid} can in particular be applied to binary matroids, i.e., matroids that can be represented over $\textsf{GF}(2)$.
These have been extensively studied~\cite{oxley2006matroid} and notably include graphic matroids.
They enjoy the following property:

\begin{theorem}\emph{\cite[Theorem 10]{wild1994theory}} \label{thm:binary-matroid}
In a simple (i.e., standard) binary matroid, a closed set is essential if and only if it is a closed circuit.
\end{theorem}

Given that all circuits of a matroid are incomparable, it follows that all essential sets of a binary matroid are incomparable.
Therefore:

\begin{corollary} \label{cor:binary-valid}
If $(\U, \cl)$ is a closure space associated to a binary matroid, then its canonical and $E$-base are equal.
Hence, its $E$-base is valid.
\end{corollary}

\section{Embedding into lattices with valid \mtt{$E$}{E}-base} \label{sec:sublattice}

In this section, we prove that any lattice can be embedded as a sublattice of a lattice with valid $E$-base.
This implies that lattices with valid $E$-base cannot be characterized by forbidden sublattices or universal sentences.

Beforehand, we give an informal presentation of the construction we shall use and the intuition on which it relies.
Let $(\U, \cl)$ be a closure space and assume that its $E$-base is not valid.
Then, following Section~\ref{sec:E-base}, there are a number of faulty essential sets, with their associated faulty pseudo-closed sets, in $(\U, \cl)$.
Our strategy consists in fixing these essential sets iteratively. 
In order to avoid all ways $C$ can be faulty, a number of which have been exhibited in Examples~\ref{ex:leaf} to \ref{ex:S7-relaxed}, it is sufficient to guarantee that the $E$-base is as exhaustive as possible with respect to $C$: every $\cl^b$-minimal spanning set of $C$ is a $E$-generator of a number of almost prime elements that generate all remaining elements (of $C$) via binary implications. 

To obtain this situation, we will \emph{lift} $C$ by introducing, for each predecessor $C'$ of $C$ a new closed set $F$ in between $C$ and $C'$ so that $F$ is almost prime w.r.t.\ $C$ in the new closure system.
This operation of inserting a single closed in between a closed set and one of its predecessor corresponds to a particular case of one-point extension~\cite[Chapter IV]{gratzer2011lattice} in the lattice theoretic set-up.
Equivalently, and more visually, it consists in subdividing each edge  $(C', C)$ of the Hasse diagram of $(\cs, \subseteq)$ as shown in Figure~\ref{fig:lifting-scheme}.
It is worth mentioning that this construction is not equivalent to doubling closed sets.
Indeed, if one wants to double a closed set $C$, then the new closed set $C'$ will lie below all successors of $C$, while here each closed set $C'$ being inserted lies in between $C$ and only one of its successors.

\begin{figure}[ht!]
    \centering
    \includegraphics[scale=\FIGliftscheme]{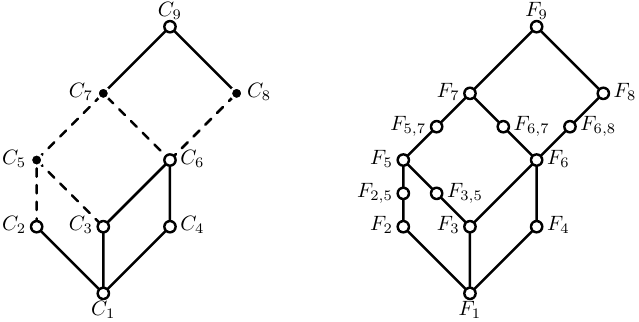}
    \caption{Scheming lifting operation. On the left, a lattice where the three closed sets $C_5$, $C_7$ and $C_8$ (in black) are to be lifted. Lifting consists in inserting a new closed set in between these closed sets and each of their predecessors, or equivalently, in subdividing the (dashed) edges $(C_2, C_5)$, $(C_3, C_5)$, $(C_5, C_7)$, $(C_6, C_7)$ and $(C_6, C_8)$ of the diagram. 
    The resulting lattice is pictured on the right. Observe that the lattice on the left is isomorphic to the sublattice consisting only of the $F_i$'s on the right.}
    \label{fig:lifting-scheme}
\end{figure}

We proceed to formally describe the lifting operation for arbitrary closed sets in our language of closure systems.
We fix a standard (non-trivial) closure space $(\U, \cl)$ with closure system $(\U, \cs)$.
We consider an indexing $\emptyset = C_1, C_2, \dots, C_m = \U$ of the closed sets of $\cs$ such that for all $1 \leq i, j \leq m$, $C_i \subseteq C_j$ entails $i \leq j$.
Let $\cc{L} \subseteq \cc{C} \setminus \{\emptyset\}$ be the set of closed sets to lift.
For each $C_j \in \cc{L}$ and each predecessor $C_i$ of $C_j$, we create a new element $y_{i,j}$.
It will be used as an almost prime element whose closure will lie in between $C_i$ and $C_j$.
Let
\[ Y = \{ y_{i,j} \st \text{$C_i$ is a predecessor of $C_j$ in $\cs$ and $C_j \in \cc{L}$}\} \] 
and let $T = \U \cup Y$ be a new ground set. 

For convenience, we introduce an intermediate mapping $f$ defined for all $Z \subseteq \U$ by $f(Z) = \{y_{i,j} \st \text{$C_j \subseteq Z$ and $C_j \in \cc{L}$} \}$.
Observe that if $Z$ does not include any closed set of $\cc{L}$, then $f(Z) = Z$. 
We also define three disjoint families of subsets of $T$:
\begin{align*}
    \cc{F}_{\cc{L}} & = \{C_i \cup f(C_i) \st C_i \in \cc{L}\} \\
    \cc{F}_{\cs \setminus \cc{L}} & = \{C_i \cup f(C_i) \st C_i \in \cs \setminus \cc{L}\} \\ 
    \cc{F}_Y   & = \{C_i \cup f(C_i) \cup \{y_{i,j}\} \st C_i \in \cs, y_{i,j} \in Y\}
\end{align*}
Let $\cc{F}_\cs = \cc{F}_{\cc{L}} \cup \cc{F}_{\cs \setminus \cc{L}}$.
We will use $F_i$ to denote the set $C_i \cup f(C_i)$ in $\cc{F}_\cs$ and $F_{i, j}$ to denote the set $C_i \cup f(C_i) \cup \{y_{i,j}\}$ in $\cc{F}_Y$.
It follows that $F_{i,j} = F_i \cup \{y_{i,j}\}$.
Finally, we put $\cc{F} = \cc{F}_\cs \cup \cc{F}_Y$.
Intuitively, $\cc{F}_Y$ models the closure of the almost prime elements we insert in between each covering pair of closed sets of $\cc{L}$, while $\cc{F}_\cs$ models the update of $\cs$ according to $Y$.
This concludes the construction of the set system $(T, \cc{F})$.

In what follows, we first need to show that $(T, \cc{F})$ is indeed a standard closure system whose closure lattice $(\cc{F}, \subseteq)$ embeds $(\cc{C}, \subseteq)$ and corresponds to the construction we hinted, pictured in Figure~\ref{fig:lifting-scheme}.
This constitutes Lemmas~\ref{lem:F-closure-system}, \ref{lem:F-subdivision}, Proposition~\ref{prop:prop-psi} and Corollary~\ref{cor:psi-standard}.
\ifx\arxiv\undefined
The proofs of these intermediate claims are routine checks.
For this reason they are omitted, but they can be find in the arXiv version of this paper.
\else
The proofs of these intermediate claims being mostly technical, they are given in Appendix~\ref{app:proofs} for readability.
\fi
This allows us to focus on our second step, namely, that the closed sets that have been lifted are now essential in $(T, \cc{F})$ but not faulty.

\ifx\arxiv\undefined
\begin{lemma} \label{lem:F-closure-system}
The pair $(T, \cc{F})$ is a closure system.
\end{lemma}
\else
\begin{lemma}[restate=LEMclosuresystem, label=lem:F-closure-system]
The pair $(T, \cc{F})$ is a closure system.
\end{lemma}
\fi

Since $(T, \cc{F})$ is a closure system, it has a corresponding closure operator which we call $\psi$.
While the full description of $\psi$ will not be necessary, we still need some of its properties that we give in the next proposition.

\ifx\arxiv\undefined
\begin{proposition} \label{prop:prop-psi}
The following statements hold:
\begin{enumerate}[(1)]
    \item for any $Z \subseteq \U$, $\psi(Z) = \cl(Z) \cup f(\cl(Z))$
    \item for any $y_{i,j} \in Y$, $\psi(y_{i,j}) = F_{i,j}$
\end{enumerate}
\end{proposition}
\else
\begin{proposition}[restate=PROPproppsi, label=prop:prop-psi]
The following statements hold:
\begin{enumerate}[(1)]
    \item for any $Z \subseteq \U$, $\psi(Z) = \cl(Z) \cup f(\cl(Z))$
    \item for any $y_{i,j} \in Y$, $\psi(y_{i,j}) = F_{i,j}$
\end{enumerate}
\end{proposition}
\fi

As a first consequence of Proposition~\ref{prop:prop-psi}, we obtain that $(T, \psi)$ is standard.

\ifx\arxiv\undefined
\begin{corollary} \label{cor:psi-standard}
The closure space $(T, \psi)$ is standard.
\end{corollary}
\else
\begin{corollary}[restate=CORpsistandard, label=cor:psi-standard]
The closure space $(T, \psi)$ is standard.
\end{corollary}
\fi

Our next step is to show that the closure system $(T, \cc{F})$ does actually correspond to the operation of inserting a closed set in between each pair of closed sets $C_i, C_j$ such that $C_i$ is a predecessor of $C_j$ and $C_j \in \cc{L}$.
To do so, we first characterize the covering relation of $(\cc{F}, \subseteq)$.

\ifx\arxiv\undefined
\begin{lemma} \label{lem:F-subdivision}
In the closure lattice $(\cc{F}, \subseteq)$, the following statements hold true:
\begin{enumerate}[(1)]
    \item for any $F_{i,j} \in \cc{F}_Y$, $F_i$ is its unique predecessor and $F_j$ its unique successor
    \item for any $F_j \in \cc{F}_\cc{L}$, the predecessors of $F_j$ are precisely the closed sets $F_{i,j}$ 
    \item for any $F_j \in \cc{F}_{\cs \setminus \cc{L}}$ different from $\emptyset$, the predecessors of $F_j$ are precisely the closed sets $F_i$ such that $C_i$ is a predecessor of $C_j$ in $\cs$
\end{enumerate}
\end{lemma}
\else
\begin{lemma}[restate=LEMFsubdivision, label=lem:F-subdivision]
In the closure lattice $(\cc{F}, \subseteq)$, the following statements hold true:
\begin{enumerate}[(1)]
    \item for any $F_{i,j} \in \cc{F}_Y$, $F_i$ is its unique predecessor and $F_j$ its unique successor
    \item for any $F_j \in \cc{F}_\cc{L}$, the predecessors of $F_j$ are precisely the closed sets $F_{i,j}$ 
    \item for any $F_j \in \cc{F}_{\cs \setminus \cc{L}}$ different from $\emptyset$, the predecessors of $F_j$ are precisely the closed sets $F_i$ such that $C_i$ is a predecessor of $C_j$ in $\cs$
\end{enumerate}
\end{lemma}
\fi

From Lemma~\ref{lem:F-subdivision}, we have that each closed set $F_{i,j} \in \cc{F}_Y$ is both join and meet-irreducible in $(\cc{F}, \subseteq)$.
As a by product, for any $F_i, F_j \in \cc{F}_\cs$, both $F_i \mt F_j$ and $F_i \jn F_j$ belongs to $\cc{F}_\cs$.
This makes $(\cc{F}_\cs, \subseteq)$ a sublattice of $(\cc{F}, \subseteq)$.
Besides, we have by definition of $\cc{F}_\cs$ that the restriction of $\psi$ to $\cs$ is an order-embedding bijection between $(\cc{C}, \subseteq)$ and $(\cc{F}_\cs, \subseteq)$.
Therefore, the two lattices $(\cc{C}, \subseteq)$ and $(\cc{F}_\cs, \subseteq)$ are isomorphic.
Along with Lemma~\ref{lem:F-subdivision}, this shows that $(\cc{F}, \subseteq)$ indeed results from inserting closed sets in between each closed set of $\cc{L}$ and its predecessors.
In particular, we get

\begin{corollary} \label{cor:C-sublattice-F}
The lattice $(\cs, \subseteq)$ is a sublattice of $(\cc{F}, \subseteq)$.
\end{corollary}

Our next step is to show that the closed sets in $\cc{F}_\cc{L}$, corresponding to the closed sets of $\cc{L}$ that have been lifted, are essential but not faulty.
Let $(T, \is_E)$ be the $E$-base of $(T, \psi)$.
Recall that all valid, non-trivial, binary implications of $(T, \cc{F})$ are included in $\is_E$.
Hence, if $\psi(Z) = \psi(z)$ for some $Z \subseteq T$ and some $z \in T$, then $z \in Z$ and $\is_E^b(Z) = \psi(Z)$ will hold true.
Consequently, we need only focus on the closed sets of $\cc{F}_\cc{L}$ that are non-ji.
In the next lemma, we show that the predecessors of such closed sets are all associated with almost prime elements.

\begin{lemma} \label{lem:F-almostprime}
For any nonempty, non-ji $F_j \in \cc{F}_\cc{L}$, and any $y_{i, j}$ such that $F_{i, j}$ is a predecessor of $F_j$, $y_{i,j}$ is almost prime in $(\idl F_j, \subseteq)$.
Furthermore, $y_{i,j}$ does not belong to any $\psi^b$-minimal spanning set of $F_j$.
\end{lemma}

\begin{proof}
We first show that $y_{i, j}$ is almost prime in $(\idl F_j, \subseteq)$.
First, we show that it is prime in $(\idl F, \subseteq)$ for every predecessor $F$ of $F_j$ in which $y_{i,j}$ appears.
However, by Lemma~\ref{lem:F-subdivision} and by construction of $\cc{F}_Y$, the unique such predecessor of $F$ is $F_{i, j}$.
Given that $\psi(y_{i,j}) = F_{i,j}$ and that $F_{i,j} \setminus \{y_{i,j}\}$ is closed, we readily obtain that $y_{i,j}$ is prime in $(\idl F_{i,j}, \subseteq)$.
Now, $y_{i,j}$ is not prime in $(\idl F_j, \subseteq)$ as there exists by assumption a predecessor $F_{k,j}$ of $F_j$ distinct from $F_{i,j}$ which satisfies $F_{i,j} \subseteq F_i \jn F_{k, j} = F_j$.
We deduce that $y_{i,j}$ is indeed almost prime in $(\idl F_j, \subseteq)$.

We now argue that no $\psi^b$-minimal spanning set of $F_j$ contains $y_{i,j}$.
Let $Z$ be a spanning set of $F_j$ containing $y_{i,j}$.
We have $F_i \subseteq \psi^b(Z)$ since $\psi^b(y_{i,j}) = \psi(y_{i,j}) = F_{i,j}$ by Proposition~\ref{prop:prop-psi}.
Let $Z' = \psi^b(Z) \setminus \{y_{i,j}\}$.
We have $\psi(Z') \subseteq F_j$.
Since $\psi(Z) = F_j$ and $F_i \subseteq \psi^b(Z)$, there exists $z \in Z$ such that $z \nsubseteq F_i$.
Therefore, $F_i \subset F_i \jn \psi(z) \subseteq F_j$.
Since $\psi(y_{i,j})$ is and $F_i, \psi(z) \neq \psi(y_{i,j})$, we have $F_i \jn \psi(z) \neq \psi(y_{i,j}) = F_{i, j}$.
Given that $F_i \cup \jn \psi(z) \subseteq Z'$, we deduce $\psi(Z') = F_i \jn \psi(z) = F_j$ By Lemma~\ref{lem:F-subdivision}.
\end{proof}

Lemmas~\ref{lem:E-generator} and \ref{lem:F-almostprime} yields the subsequent corollary:

\begin{corollary} \label{cor:F-Egen}
Let $F_j \in \cc{F}_\cc{L}$ and let $A$ be a $\psi^b$-minimal spanning set of $F_j$ with at least $2$ elements, if any.
Then, for any $y_{i,j}$ such that $F_{i,j}$ is a predecessor of $F_j$, $A$ is a $E$-generator of $y_{i,j}$.
\end{corollary}

Now let $F_j \in \cc{F}_\cc{L}$ be non-ji.
By Lemmas~\ref{lem:E-QC} and \ref{lem:F-almostprime}, $F_j$ is essential.
Hence the canonical base of $(T, \psi)$ contains some implicationsof the form $P \imp F_j \setminus P$ where $P$ is a pseudo-closed set spanning $F_j$.
Since all binary implications of $(T, \psi)$ are in $(T, \is_E)$, $\is_E(P)$ contains a $\psi^b$-minimal spanning set $A$ of $F_j$.
By Corollary~\ref{cor:F-Egen}, $A \imp y_{i, j}$ belongs to $\is_E$ for every $y_{i,j}$ such that $F_{i, j}$ is a predecessor of $F_j$.
Given that $\psi(y_{i,j}) = F_{i,j}$, $F_{i,j} \subseteq \is_E(P)$ thus holds.
As the predecessors of $F_j$ are precisely these $F_{i,j}$'s by Lemma~\ref{lem:F-subdivision} and as moreover $F_j = \bigcup F_{i, j}$ since $F_j$ is neither empty nor join-irreducible, we obtain that $\is_E(P) = \psi(P) = F_j$, that is, that $P \imp F_j \setminus F_j$ does follow from $(T, \is_E)$.
This leads to the following lemma:

\begin{lemma} \label{lem:F-essential-notfaulty}
In $(T, \psi)$ all the closed sets in $\cc{F}_\cc{L}$ are essential but not faulty.
\end{lemma}

Henceforth, if $\cc{L}$ comprises the essential sets of $(\U, \cl)$ that are faulty, these essential sets are fixed in $(T, \psi)$.
Moreover, $\cc{F}_Y$ only contains join-irreducible closed sets of $(\cc{F}, \subseteq)$, so that the closed sets that are introduced cannot be faulty.
Hence, one could expect that in $(T, \psi)$, the $E$-base is valid, as in the result of lifting the faulty essential set of the closure space of Example~\ref{ex:leaf}, given in Figure~\ref{fig:leaf-lift}.
\begin{figure}[ht!]
    \centering
    \includegraphics[scale=\FIGleaflift]{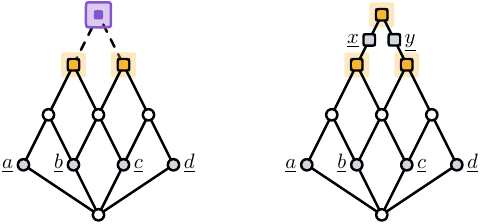}
    \includegraphics[width=\textwidth, page=3]{figures/captions.pdf}
    \caption{Lifting the faulty essential set of the closure space of Example~\ref{ex:leaf}.
    For readability, we only label ji closed sets with the element they are the closure of.
    The result of lifting the faulty essential set $abcd$ is pictured on the right, where we call $x, y$ the two new elements of the ground set.
    The (aggregated) $E$-base of the resulting closure space $(\{a, b, c, d, x, y\}, \psi)$ equals $\{x \imp abc, y \imp bcd, ac \imp b, bd \imp c, ad \imp xy\}$.
    The only implication of the canonical base of $(\{a, b, c, d, x, y\}, \psi)$ not in the $E$-base is $ad \imp xybc$ which follows from $ad \imp xy$, $x \imp abc$, $y \imp bcd$.
    The $E$-base is thus valid.}
    \label{fig:leaf-lift}
\end{figure}
This is however not true, because lifting some closed sets might turn a closed set of $\cs$ which is not lifted into a faulty essential set, even if it was not faulty, let alone essential, before.
This is illustrated in Figure~\ref{fig:faulty-lifting}.
\begin{figure}[ht!]
\centering 
\includegraphics[page=1, width=\textwidth]{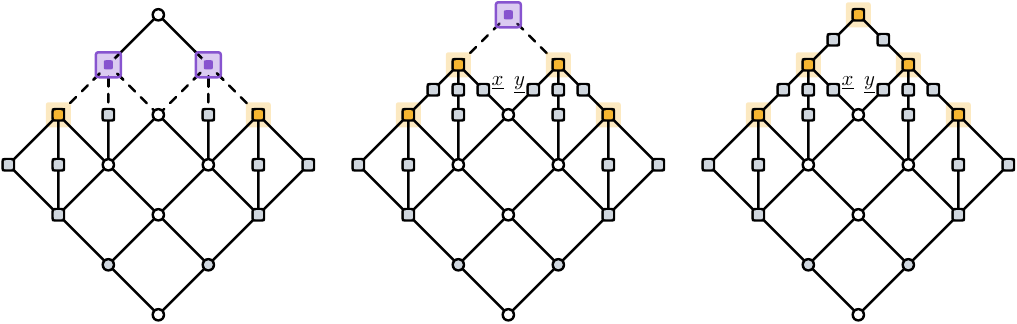}
\includegraphics[width=\textwidth, page=3]{figures/captions.pdf}
\caption{The closure lattice on the left does not have valid $E$-base (see Theorem~\ref{thm:modular-Ebase}) and has two faulty essential sets. However, lifting the faulty essential sets makes the top closed set of the new closure lattice faulty: $xy$ induces a pseudo-closed set that is not subsumed by any $E$-generator, since no other element is almost prime.
Thus, we need to reapply lifting.
Observe that the second lifting operation involves only a closed set strictly above those that have been previously lifted.}
\label{fig:faulty-lifting}
\end{figure}
Still, we show in the next statement that a new faulty essential closed set must properly contain a closed set that has been lifted.

\begin{lemma} \label{lem:faulty-increasing-lift}
An essential closed set $F_j \in \cc{F}$ is faulty in $(T, \psi)$ only if $F_j \in \cc{F}_{\cs \setminus \cc{L}}$ and there exists $F_i \in \cc{F}_\cc{L}$ such that $F_i \subset F_j$.
\end{lemma}

\begin{proof}
The fact that $F_j$ belongs to $\cc{F}_{\cs \setminus \cc{L}}$ follows from previous discussion.
Now assume for contradiction that $F_j$ does not subsume any $F_i \in \cc{F}_\cc{L}$.
Then, by definition of $f$, $F_j = C_j$ and hence, $\idl_{\cs} C_j = \idl_{\cc{F}} F_j$.
Consequently, $C_j$ is already faulty in $(\U, \cl)$.
Therefore, $C_j \in \cc{L}$ and $F_j \in \cc{F}_\cc{L}$, a contradiction.
\end{proof}

We are ready to conclude.
Let $T_0 = \U$ and $\psi_0 = \cl$ and consider the algorithm that, at step $i$, produces the closure space $(T_i, \psi_i)$ by lifting all the inclusion-wise minimal faulty essential sets of $(T_{i-1}, \psi_{i-1})$, while there are some.
Since lifting a closed set $C_j$ introduces elements below $C_j$ only, and as $\cc{L}$ comprises incomparable closed sets, it follows from Lemma~\ref{lem:faulty-increasing-lift} that the maximal size of a maximal chain from $T_{i}$ to a faulty essential closed set of $(T_i, \psi_i)$ strictly decreases as compared to $(T_{i-1}, \psi_{i-1})$.
As we are considering finite closure spaces the procedure will stop within at most $\card{\U}$ steps. 
The resulting closure space $(T, \psi)$ does not have faulty essential sets anymore, and hence has valid $E$-base. 
Moreover, $(\cs, \subseteq)$ is a sublattice of $(\cc{F}, \subseteq)$ since we can apply Corollary~\ref{cor:C-sublattice-F} transitively.
We finally obtain Theorem~\ref{thm:sublattice}:

\ifx\arxiv\undefined
\begingroup
\def\thetheorem{\ref{thm:sublattice}}
\begin{theorem}
For any standard closure space $(\U, \cl)$ with lattice $(\cs, \subseteq)$, there exists a standard closure space $(T, \psi)$ with lattice $(\cc{F}, \subseteq)$ such that $(\cs, \subseteq)$ is a sublattice of $(\cc{F}, \subseteq)$ and $(T, \psi)$ has valid $E$-base.
\end{theorem}
\addtocounter{theorem}{-1}
\endgroup
\else
\THMsub*
\fi

\begin{corollary} \label{cor:Ebase-universal}
Lattices with valid $E$-base cannot be characterized by universal sentences or by forbidden sublattices.
\end{corollary}

To conclude the section, observe that lifting all the closed sets of $\cs$ but $\emptyset$ all at once, i.e., putting $\cc{L} = \cs \setminus \{\emptyset\}$, also provides a closure space with valid $E$-base.
In fact, in this case, the $E$-base of $(T, \psi)$ will be as large as it can be: for every non-ji closed set $F$, and every $\psi^b$-minimal spanning set $A$ of $F$, $A$ will be a $E$-generator of almost prime elements that, using binary implications, will directly reach all remaining elements of $F$.
We observe though that, as compared to our initial approach, lifting all closed sets leads to lift unneccessary closed sets such as ji ones or closed sets that do not subsume any essential closed sets.

\section{Discussion} \label{sec:conclusion}

The study of the $E$-base and the classes of lattices where it is valid was initiated by Adaricheva et al.~\cite{adaricheva2013ordered}, who showed that, even though it is not valid in general, it is for closure spaces without $D$-cycles, i.e., lower bounded lattices.
In this contribution, we continued this line of work using the canonical base and its properties as our main weapon. 
We investigated three classical generalizations of Boolean and/or distributive lattices being modular, geometric, and semidistributive lattices, as well as the class of lattices with valid $E$-base.
To conclude, we discuss possible directions for future research, and propose a couple of open questions in this regard.

In Section~\ref{sec:E-base} we gave examples exhibiting different degrees to which the $E$-base can fail to be valid.
An essential set can be faulty because of (at least) the following three increasingly demanding reasons: it is not spanned by any $E$-generator, one of its pseudo-closed set is not subsumed by an $E$-generator, or the $E$-generators included in the pseudo-closed sets do not allow to recover the essential set properly.
This distinction suggests that even though a class of lattices might fail to have valid $E$-base in general, it can still obey some degree of validity. 
This leads to our first two questions:
\begin{question}
Are there classes of lattices where the $E$-base is not valid in general, but where essential sets are always spanned by some $E$-generator?
\end{question}
\begin{question}
Are there classes of lattices where the $E$-base is not valid in general, but the right-closure (replacing implications $A \imp B$ by $A \imp \cl(A)$) of the $E$-base is valid?
\end{question}
Notice that join-distributive, upper-bounded, geometric, and modular lattices are not such classes as proved by Examples~\ref{ex:fail-jdis}, \ref{ex:fullfail-geom} and \ref{ex:fail-mod} respectively.
An interesting way to tackle this question would be to craft maps that, contrarily to Section~\ref{sec:sublattice}, embed a lattice with invalid $E$-base in a lattice of the same class that has an essential set not spanned by $E$-generator.

In Sections~\ref{sec:semidistributive}, \ref{sec:geometric}, and \ref{sec:modular}, we proved that the $E$-base of semidistributive lattices is valid and characterized the modular and geometric lattices with valid $E$-base.
This study paves the way for further explorations of other classes of lattices such as join-distributive lattices or convex geometries for instance.
Within the class of semidistributive lattices, those without $D$-cycles are also covered by the result established in \cite{adaricheva2013ordered}.
Therefore, the new cases for which $E$-base is valid are closure spaces that are semidistributive and possess $D$-cycles.  
For instance in Example \ref{ex:carpet}, we have $D$-cycle $c D f D g D d D c$. 
The non-binary implications in the $E$-base have exactly these elements and their $D$-generators $A_k$.
Other examples confirm this observation, which brings us to the following question.

\begin{question}
Does the $E$-base of a closure space with semidistributive lattice and $D$-cycles always include implications $A_k\to x_k$, where some $x_{k-1}\in A_k$ and $x_{k-1}, x_k$ are consecutive members of some $D$-cycle?
\end{question}

Besides, geometric lattices seem to be an interesting class to study further.
Even though we characterized those geometric lattices that have valid $E$-base, it would be interesting to identify more classes of matroid, on top of binary ones, that fall into the scope of Theorem~\ref{thm:matroid-valid}.

\begin{question} \label{ques:matroid}
What are the matroids for which $E$-base is valid?
\end{question}

Given Theorem~\ref{thm:matroid-valid}, answering Question~\ref{ques:matroid} leads to the question of characterizing essential sets of geometric lattices from the perspective of matroids.

\begin{question}
Is there a characterization of essential sets of geometric lattices in terms of bases or circuits of matroid?
\end{question}

While there exists some preliminary results on the question~\cite{duquenne1991core,wild1994theory}, there is no full characterization to our knowledge.

More generally, upper-semimodular lattices are a natural generalization of both geometric and modular lattices.
Hence, it is also a natural candidate for future investigations.
For the moment though we observe that none of the characterization we used for geometric or modular lattices does actually characterize upper-semimodular lattices with valid $E$-base, as shown in the next example.

\begin{example} \label{ex:config-usm}
The closure space associated whose associated lattice is given in Figure~\ref{fig:usm-config} is upper-semimodular and has valid $E$-base.
Its canonical base and $E$-base are given by:
\begin{align*}
\is_{DG} & = \{f \imp a, d \imp c, e \imp c\} & (\text{binary}) \\ 
& \cup \{\agset{ab \imp c}, \agset{ac \imp b}, \agset{bc \imp a}\} & (\agset{abc}) \\
& \cup \{\agset{abcf \imp de}, \agset{abcd \imp ef}, \agset{abce \imp df}, \agset{cde \imp bcf} \} & (\agset{abcdef}) \\
& & \\ 
\is_E & = \{f \imp a, d \imp c, e \imp c\} & (\text{binary}) \\ 
& \cup \{\agset{ab \imp c}, \agset{ac \imp b}, \agset{bc \imp a}\} & (\agset{abc}) \\
& \cup \{\agset{ad \imp ef}, \agset{ae \imp df}, \agset{bd \imp ef}, \agset{be \imp df}, \agset{bf \imp de}, \agset{cf \imp de}, \agset{de \imp f} \} & (\agset{abcdef})
\end{align*}
Observe that $\agset{abc}$ and $\agset{abcdef}$ are comparable essential sets.
Moreover, for any predecessor $C$ of $\agset{abcdef}$, we have $\card{\agset{abcdef} \setminus C} > 1$.
\begin{figure}[ht!]
    \centering
    \includegraphics[scale=\FIGusm]{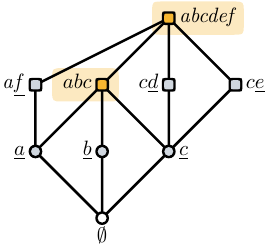}
    
    \includegraphics[page=4]{figures/captions.pdf}
    \caption{An upper-semimodular lattice with valid $E$-base that has comparable essentials, and where the (set-)difference between an essential set and one of its predecessors can be greater than one.}
    \label{fig:usm-config}
\end{figure}
\end{example}

Finally, we proposed in Section~\ref{sec:sublattice} a construction that embeds any lattice into a lattice with valid $E$-base by iteratively lifting faulty essential sets.
The lifting operation consists in inserting a new closed set between the essential set and each of its predecessor.
However, if one selects in a more surgical way the predecessors to which apply the operation, possibly according to the pseudo-closed sets that are not mapped to their closure in the $E$-base, then one might end up with a lattice with valid $E$-base that has fewer closed sets.
This idea is illustrated in Figure~\ref{fig:optimal-lift} where we improve the lifting of the lattice in Figure~\ref{fig:faulty-lifting}.
We are thus led to the following question:

\begin{question}
Given a lattice $(\cs, \subseteq)$, what is the minimal possible size of a lattice $(\cc{F}, \subseteq)$ such that $(\cs, \subseteq)$ is a sublattice of $(\cc{F}, \subseteq)$ and $(\cc{F}, \subseteq)$ has valid $E$-base?
\end{question}

\begin{figure}[ht!]
    \centering
    \includegraphics[page=2, scale=\FIGminlift]{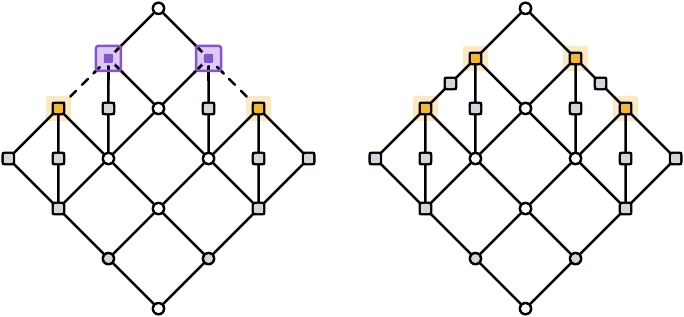}
    \includegraphics[width=\textwidth, page=2]{figures/captions.pdf}
    \caption{Lifting only some predecessors of a faulty essential set can lead to a smaller extension with valid $E$-base.}
    \label{fig:optimal-lift}
\end{figure}

As a second observation, our construction does not, in general, preserve semidistributive or semimodular laws.
We give in Figure~\ref{fig:grid-lift} an example where all four laws are lost during lifting.
This leads to wonder about embeddings that would preserve the class of lattice at hand:

\begin{question}
If a class of lattices does not have valid $E$-base in general, does it admit a lattice embedding that embeds any lattice with non-valid $E$-base in a lattice with valid one?
\end{question}

\begin{figure}[ht!]
    \centering
    \includegraphics[scale=\FIGgridlift]{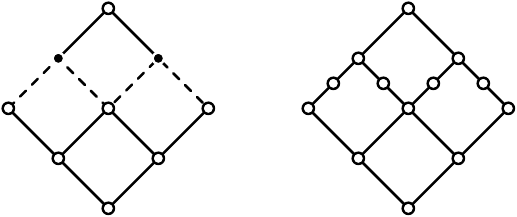}
    \caption{On the left a distributive lattice with two elements to lift (in black). 
    The resulting lattice, on the right, fails upper- and lower-semimodularity as well as join- and meet-semidistributivity.}
    \label{fig:grid-lift}
\end{figure}

\bibliographystyle{alphaurl}
\bibliography{biblio}

\appendix

\section{Proofs of Section~\ref{sec:sublattice}} \label{app:proofs}

\LEMclosuresystem*

\begin{proof}
We have $F_m = C_m \cup f(C_m) = T \in \cc{F}_\cs$, hence $T \in \cc{F}$, by definition of $\cc{F}$.
Now we prove that the intersection of any two incomparable closed sets of $\cc{F}$ lies in $\cc{F}$.
We have three cases to consider depending on whether the closed sets belong to $\cc{F}_\cs$ or $\cc{F}_Y$.
We show that it is in fact sufficient to consider only closed sets in $\cc{F}_\cs$, i.e., that for every distinct (and incomparable) $F_{i,j}, F_{k,\ell} \in \cc{F}_Y$ one has:
\[
    F_{i,j} \cap F_{k,\ell} = F_{i,j} \cap F_k = F_i \cap F_{k,\ell} = F_i \cap F_k
\]
Recalling that $F_{i,j} = F_i \cup \{y_{i,j}\}$ and $F_{k,\ell} = F_k \cup \{y_{k,\ell}\}$, we have:
\begin{align*}
    F_{i,j} \cap F_{k,\ell} & = (F_i \cup \{y_{i,j}\}) \cap (F_k \cup \{y_{k,\ell}\}) \\ 
    & = (F_i \cap F_k) \cup  (F_i \cap \{y_{k,\ell}\}) \cup (F_k \cap \{y_{i,j}\}) \cup (\{y_{i,j}\} \cap \{y_{k,\ell}\})
\end{align*}
Since $F_{i,j}$ and $F_{k,\ell}$ are distinct, $y_{i, j}$ and $y_{k,\ell}$ are also distinct, and $\{y_{i,j}\} \cap \{y_{k,\ell}\} = \emptyset$ holds.
Then, suppose for contradiction that $(F_i \cap \{y_{k,\ell}\}) \neq \emptyset$, i.e., that $y_{k,\ell} \in f(C_i)$.
We deduce that $C_\ell \subseteq C_i$ and hence $F_{k,\ell} \subseteq F_i \subseteq F_{i,j}$, a contradiction with $F_{i,j}$ and $F_{k,\ell}$ being incomparable.
It follows that $F_i \cap \{y_{k,\ell}\} = \emptyset$ and similarly $(F_k \cap \{y_{i,j}\}) = \emptyset$.
We derive $F_{i,j} \cap F_{k,\ell} = F_{i,j} \cap F_k = F_i \cap F_{k,\ell} = F_i \cap $ as expected.
Now, 
\begin{align*}
    F_i \cap F_k & = (C_i \cup f(C_i)) \cap (C_k \cup f(C_k)) & \\ 
    & = (C_i \cap C_k) \cup (f(C_i) \cap f(C_k)) & \text{since $C_i \cap f(C_k) = C_k \cap f(C_i) = \emptyset$} \\ 
    & = (C_i \cap C_k) \cup f(C_i \cap C_k) & \text{as $y_{p, q} \in f(C_i) \cap f(C_k)$ iff $C_q \subseteq C_i \cap C_k$}
\end{align*}
From $C_i \cap C_k \in \cs$, we obtain $F_i \cap F_k \in \cs$, and we conclude that $(T, \cc{F})$ is indeed a closure system.
\end{proof}

\PROPproppsi*

\begin{proof}
We prove the items in order.

\textbf{Item 1.} Let $Z \subseteq \U$. 
Observe that $\cl(Z) \cup f(\cl(Z)) \in \cc{F}$ holds by definition of $\cc{F}$.
Now, let $F \in \cc{F}$ such that $Z \subseteq F$.
We prove that $\cl(Z) \cup f(\cl(Z)) \subseteq F$.
We can assume $F \in \cc{F}_\cs$ as $F \in \cc{F}_Y$ entails $Z \subseteq F \setminus \{y_{i,k}\}$ where $F = F_{i,k} = F_i \cup \{y_{i,k}\}$.
By definition of $F$, there exists $C \in \cs$ such that $F = C \cup f(C)$.
Therefore, $\cl(Z) \subseteq C$ holds as $\cl(Z) \subseteq F \cap \U$.
This completes the proof that $\psi(Z) = \cl(Z) \cup f(\cl(Z))$. 

\textbf{Item 2.} Let $y_{i,j} \in Y$.
We show that $\psi(y_{i,j}) = F_{i,j}$ using double-inclusion.
The $\subseteq$ part follows from the definition of $\cc{F}$.
To prove the $\supseteq$ part, we show that $F_{i,j} \subseteq F$ for any $F \in \cc{F}$ such that $y_{i,j} \in F$.
Consider such a $F$.
Either $F \in \cc{F}_\cs$ or $F \in \cc{F}_Y$.
In the case $F \in \cc{F}_\cs$, $F = F_k = C_k \cup f(C_k)$ for some $C_k \in \cs$.
Then, $y_{i,j} \in F$ implies $C_j \subseteq C_k$ by definition of $f$.
Therefore, $F_{i, j} \subseteq C_j \cup f(C_j) \subseteq C_k \cup f(C_k) = F_k$ holds.
As for the case where $F \in \cc{F}_Y$, there exists $y_{k,\ell}$ such that $F = F_k \cup f(C_k) \cup \{y_{k,\ell}\}$.
If $y_{k,\ell} = y_{i,j}$ then $F = F_{i,j}$ readily holds, otherwise $y_{i,j} \in F_k$ and the previous reasoning applies.
We deduce that $\psi(y_{i,j}) = F_{i,j}$ as required.
\end{proof}

\CORpsistandard*

\begin{proof}
We show that for any $z \in T$, $\psi(z) \setminus \{z\} \in \cc{F}$.
We have two cases: $z \in Y$ or $z \in \U$.
We start with the case $z \in Y$.
Let $z = y_{i,j}$ for some $y_{i,j} \in Y$.
By Proposition~\ref{prop:prop-psi}, $\psi(y_{i,j}) = F_i \cup \{y_{i,j}\}$.
As $F_i \in \cc{F}$, we directly obtain $\psi(y_{i,j}) \setminus \{y_{i,j}\} \in \cc{F}$.

We consider now the case when $z \in \U$.
We have $\psi(z) = \cl(z) \cup f(\cl(z))$ by Proposition~\ref{prop:prop-psi}.
Let us put $C_j = \cl(z)$.
Given that $(\U, \cs)$ is standard, $C_j$ is ji in $(\cs, \subseteq)$, and hence it has a unique predecessor $C_i = C_j \setminus \{z\}$.
We have two cases, $C_j \in \cc{L}$ or $C_j \notin \cc{L}$.
In the first case, we obtain $\psi(z) = C_j \cup f(C_j) = C_j \cup f(C_i) \cup \{y_{i,j}\}$.
Then, we deduce $(C_j \cup f(C_j)) \setminus \{z\} = (C_j \setminus \{z\}) \cup f(C_i) \cup \{y_{i,j}\} = F_{i,j}$ with $F_{i,j} \in \cc{F}$.
If on the other hand $C_j \notin \cc{L}$, then $\psi(z) = C_j \cup f(C_i) = F_i \cup \{z\}$ with $F_i \in \cc{F}$.
Since for any $z$, $\psi(z) \setminus \{z\} \in \cc{F}$, we conclude that $(T, \cc{F})$ is standard as required.
\end{proof}

\LEMFsubdivision*

\begin{proof}
We prove the items in order.

\textbf{Item 1.}
The fact that $F_i$ is the unique predecessor of $F_{i,j}$ follows from the definition of $\cc{F}_Y$ as well as from Proposition~\ref{prop:prop-psi}.
Consider $F_j$.
We have $F_{i,j} \subseteq F_j$ by definition.
We show that for any $F \in \cc{F}$, $F_{i, j} \subset F$ implies $F_j \subseteq F$.
If $F \in \cc{F}_\cs$ then $F = C_k \cup f(C_k)$ for some $C_k \in \cs$.
By definition of $f$, $y_{i,j} \in f(C_k)$ entails $F_j \subseteq F$.
If $F \in \cc{F}_Y$ and $F_{i,j} \subset F$, then $F = F_k \cup \{y_{k,\ell}\}$ for some $y_{k,\ell} \neq y_{i,j}$.
We obtain $F_j \subseteq F$ again.
This concludes the proof that $F_j$ is the unique successor of $F_{i,j}$ in $(\cc{F}, \subseteq)$. 

\textbf{Item 2.}
The fact that the $F_{i,j}$'s are predecessors of $F_j$ is a consequence of the first case.
We need to argue that there are no other predecessors.
Assume for contradiction there is another predecessor $F$ of $F_j$.
By Item 1 and by definition of $\cc{F}_Y$, $F$ cannot belong to $\cc{F}_Y$ as otherwise it would be equal to one of the $F_{i, j}$'s such that $C_i \prec C_j$. 
Hence, $F = C_k \cup f(C_k)$ for some $C_k \subseteq C_j$ in $\cs$.
As $F \subset F_j$, we moreover has $C_k \subset C_j$.
Hence, there exists $C_i \subseteq C_j$ such that $C_k \subseteq C_i \subset C_j$ and we obtain $F_k \subseteq F_{i,j} \subset F_j$.
This contradicts $F$ being a predecessor of $F_j$.

\textbf{Item 3.} 
Since $f$ is monotone, we readily have that $F_i \subseteq F_j$ for any $C_i \prec C_j$ in $\cs$.
Moreover, all the $F_i$'s are incomparable as the $C_i$'s are.
We show that any $F \in \cc{F}$ such that $F \subset F_j$ satisfies $F \subseteq F_i$ for some $F_i$ such that $C_i \prec C_j$.
If $F \in \cc{F}_\cs$, then $F = C_k \cup f(C_k)$ for some $C_k \in \cs$.
As for any $C, C' \in \cs$, $C \cup f(C) \subseteq C' \cup f(C')$ iff $C \subseteq C'$, we further obtain $C_k \subset C_j$.
Therefore, there exists $C_i \prec C_j$ in $\cs$ such that $C_k \subseteq C_i \subset C_j$ and hence $F \subseteq F_i \prec F_j$ holds.
Assume now that $F \in \cc{F}_Y$.
Then, $F = F_k \cup \{y_{k,\ell}\}$ for some $C_k \prec C_\ell$ such that $C_\ell \in \cc{L}$. 
Since $F \subset F_j$, $y_{k,\ell} \in F_j$ and $C_\ell \subseteq C_j$ follows.
As $C_\ell \in \cc{L}$ and $C_j \notin \cc{L}$ by assumption, we deduce $C_\ell \subset C_j$.
Therefore, there exists $C_i$ a predecessor of $C_j$ such that $C_\ell \subseteq C_i \subset C_j$.
We deduce $F \subset F_\ell \subseteq F_i \subset F_j$.
This concludes the proof.
\end{proof}

\end{document}